\documentclass[sn-mathphys,Numbered]{sn-jnl}
\usepackage{amsmath,amssymb,amsthm,bm,mathrsfs}
\usepackage{graphicx}
\usepackage{hyperref}
\usepackage{xcolor}
\usepackage[T1]{fontenc}
\usepackage{multirow}
\usepackage{comment}
\usepackage{amsfonts}
\usepackage[title]{appendix}
\usepackage{textcomp}
\usepackage{manyfoot}
\usepackage{listings}
\usepackage{booktabs}
\usepackage{algorithm,algorithmic}

\usepackage{aliascnt}
\theoremstyle{plain}
\newtheorem{theorem}{Theorem}
\newaliascnt{lemma}{theorem}

\aliascntresetthe{lemma}
\newaliascnt{proposition}{theorem}
\newtheorem{proposition}[proposition]{Proposition}
\aliascntresetthe{proposition}
\newaliascnt{corollary}{theorem}
\newtheorem{corollary}[corollary]{Corollary}
\aliascntresetthe{corollary}
\newaliascnt{assumption}{theorem}
\newtheorem{assumption}[assumption]{Assumption}
\aliascntresetthe{assumption}

\newtheorem{remark}{Remark}
\usepackage[capitalise,nameinlink,noabbrev]{cleveref}
\crefname{equation}{}{}
\crefname{assumption}{Assumption}{Assumptions}
\crefname{algocf}{Algorithm}{Algorithms}
\graphicspath{{fig/}}
\definecolor{bkgndcolor}{rgb}{1,1,1}
\definecolor{bkgndcolor}{rgb}{0.78,0.93,0.8}
\allowdisplaybreaks

\DeclareMathOperator*{\argmin}{\arg\min}

\DeclareMathOperator{\diag}{diag}
\DeclareMathOperator{\spec}{spec}

\DeclareMathOperator{\range}{range}

\DeclareMathOperator{\Span}{span}

\DeclareMathOperator{\order}{\mathcal{O}}

\newcommand{\defi}{:=}

\newcommand*{\set}[1]{\left\lbrace#1\right\rbrace}

\newcommand*{\Htran}{\mathsf{H}}

\newcommand{\out}{\mathsf{out}}

\newcommand{\bmat}[1]{\begin{bmatrix}#1\end{bmatrix}}

\newcommand{\abs}[1]{\lvert#1\rvert}
\newcommand{\norm}[1]{\lVert#1\rVert}
\newcommand{\bigabs}[1]{\bigl\lvert#1\bigr\rvert}

\newcommand{\bignorm}[1]{\bigl\lVert#1\bigr\rVert}
\newcommand{\Bignorm}[1]{\Bigl\lVert#1\Bigr\rVert}
\newcommand*{\md}{\mathop{}\mathopen{}\mathrm{d}}
\newcommand*{\mi}{\mathrm i}
\newcommand{\nlf}{F_{\mathsf{nl}}^{[N]}}
\newcommand{\C}{\mathbb{C}}

\newcommand{\region}{\mathcal{D}}
\newcommand{\cif}{F_{\mathsf{ci}}}
\newcommand{\itf}{F_{\mathsf{it}}}
\newcommand{\itfo}{\mathcal{F}_{\mathsf{it}}}
\newcommand{\epsnl}{\epsilon_{\mathsf{nl}}}

\newcommand{\cnl}{c_{\mathsf{nl}}}
\renewcommand{\emph}[1]{\textbf{#1}}

\title[Convergence of contour integral-based eigensolvers for NEPs]{Linear convergence of iterative contour integral-based eigensolvers for nonlinear eigenvalue problems}

\author[1]{\fnm{Daniel} \sur{Kressner}}\email{daniel.kressner@epfl.ch}
\author[2]{\fnm{Yuqi} \sur{Liu}}\email{y.liu12@leeds.ac.uk}
\author[3]{\fnm{Jose E.} \sur{Roman}}\email{jroman@dsic.upv.es}
\author[4]{\fnm{Meiyue} \sur{Shao}}\email{myshao@fudan.edu.cn}
\author*[1]{\fnm{Nian} \sur{Shao}}\email{nian.shao@epfl.ch}

\affil[1]{\orgdiv{Institute of Mathematics}, \orgname{EPFL}, \orgaddress{\city{Lausanne}, \postcode{1015},  \country{Switzerland}}}
\affil[2]{\orgdiv{School of Computer Science}, \orgname{University of Leeds}, \orgaddress{\city{Leeds}, \postcode{LS2 9JT},  \country{UK}}}
\affil[3]{\orgdiv{Departament de Sistemes Inform\`atics i Computaci\'o}, \orgname{Universitat Polit\`ecnica de Val\`encia}, \orgaddress{\city{Val\`encia}, \postcode{46022}, \country{Spain}}}
\affil[4]{\orgdiv{School of Data Science and Shanghai Key Laboratory for Contemporary Applied Mathematics}, \orgname{Fudan University}, \orgaddress{\city{Shanghai}, \postcode{200433},  \country{China}}}

\begin{document}

\abstract{ 
Solving nonlinear eigenvalue problems is an important and challenging task in scientific computing.
Contour integral-based approaches are attractive for such eigenvalue problems because they reliably target all eigenvalues in a prescribed domain.
However, unlike in the linear case, many traditional methods of this type, such as Beyn's method, lack an inherent iterative refinement mechanism. Consequently, achieving high accuracy requires high-quality quadrature rules for approximating the contour integral, which often leads to prohibitive computational costs.
A notable exception is the so-called NLFEAST algorithm, which combines contour integral techniques with a nonlinear Rayleigh--Ritz extraction step.
In this work, we propose a general framework of iterative contour integral-based methods for nonlinear eigenvalue problems that includes NLFEAST. This allows us to prove linear convergence of NLFEAST under mild assumptions and also explains why certain nonlinear eigensolvers do not combine well with iterative methods.
Numerical experiments confirm our theoretical findings; in particular that NLFEAST can achieve high accuracy even with a limited number of quadrature nodes, significantly outperforming Beyn's method on challenging problems.
}
\keywords{Nonlinear eigenvalue problems, contour integral, convergence, NLFEAST}
\pacs[AMS subject classifications (2020)]{65H17, 65F15} 
\maketitle

\section{Introduction}
\label{sec:intro}
This work is concerned with nonlinear eigenvalue problems (NEPs) \cite{GT2017,MV2004} of the form
\begin{equation}
\label{eq:nep}
T(\lambda)u=0,
\qquad u\in\C^n\setminus\set{0},
\qquad \lambda\in\region.
\end{equation}
Throughout this paper, we assume that \(\region\subset\C\) is a nonempty domain enclosed by a Jordan curve, and $T(\cdot)$ is an $n\times n$ holomorphic matrix-valued function on some domain that contains the closure of $\region$. 
We also assume that the NEP~\cref{eq:nep} is regular, that is, that  $\det T(\xi_{0})\neq 0$ for some $\xi_{0}\in\region$.
In the particular case \(T(\xi)=\xi I-A\), the NEP~\eqref{eq:nep} reduces to a standard eigenvalue problem for the matrix $A$.
NEPs arise in a broad range of scientific computing applications; see~\cite{TH2001,CER2020,LYG2019,BV2007} and the references therein.

Many algorithms for solving NEPs are based on rational approximation.
Three types of rational approximation are popular for
this purpose: Approximations to the indicator function of the
domain~\cite{Polizzi2009,SS2003,Beyn2012,HSS2016}, approximations using a surrogate function of the resolvent \cite{BST2024}, and approximations of the matrix-valued function $T(\cdot)$ itself~\cite{GNT2022,GKV2024,SEM2019,SRK2014,RKW2015}.
In this paper, we focus on the first type, commonly referred to as
contour integral-based methods \cite{BEG2023}, in which the indicator function of the target domain is approximated by a discretized contour integral.
These methods have become popular in recent years due to their relatively straightforward implementation and great potential for parallelization.

Let $\chi_{\region}$ denote the indicator function of $\region$, that is, $\chi_{\region}(\xi)=1$ if $\xi\in\region$ and $\chi_{\region}(\xi)=0$ otherwise. Trivially, this function admits the contour integral representation 
\begin{equation*}
    \chi_{\region}(z) = \frac{1}{2\pi\mi }\oint_{\partial \region}(\xi-z)^{-1}\md \xi.
\end{equation*}
For a standard eigenvalue problem $T(\xi)=\xi I-A$, inserting $A$ into this relation gives 
\begin{equation} \label{eq:indA}
    \chi_{\region}(A) = \frac{1}{2\pi\mi }\oint_{\partial \region}(\xi I-A)^{-1}\md \xi,
\end{equation}
which coincides with the spectral projector of $A$ belonging to the eigenvalues in $\region$.
More generally, we can consider higher-order moments:
\begin{equation}
\label{eq:defmon}
M_s=\frac{1}{2\pi\mi}\oint_{\partial\region}\xi^s T(\xi)^{-1}\md\xi,
\qquad s=0,1,\dotsc.
\end{equation}
Often, one can extract all relevant spectral information directly from the matrix pencil $M_{1} - \xi M_{0}$. When a large number of eigenvalues are required, in particular when this number exceeds $n$, it can be preferable to extract the spectral information from the (block) Hankel matrix pencil
\begin{equation*}
    \begin{bmatrix}
        M_{1} & M_{2} & \cdots & M_{s+1}\\ 
        M_{2} & M_{3} & \cdots & M_{s+2} \\
        \vdots & \vdots & \ddots & \vdots \\
        M_{s+1} & M_{s+2} & \cdots & M_{2s+1}
    \end{bmatrix}- \xi \cdot
    \begin{bmatrix}
        M_{0} & M_{1} & \cdots & M_{s}\\ 
        M_{1} & M_{2} & \cdots & M_{s+1} \\
        \vdots & \vdots & \ddots & \vdots \\
        M_{s} & M_{s+1} & \cdots & M_{2s}
    \end{bmatrix}.
\end{equation*}
By replacing the contour integral in~\cref{eq:defmon} with a discretized version and probing the moments, these two strategies lead to Beyn's method~\cite{Beyn2012} and the nonlinear (block) Sakurai--Sugiura (SS) method~\cite{AST2009}.
Specifically, let $M_{0,N}$ and $M_{1,N}$ be approximations of $M_{0}$ and $M_{1}$, respectively, such as those obtained from numerical quadrature with $N$ quadrature nodes. Given a generic probing matrix $X$, typically chosen as a (complex) Gaussian random matrix, Beyn's method \cite{Beyn2012} essentially computes the eigenvalues of $(M_{0,N}X)^{\dagger}(M_{1,N}X)$, where $(\cdot)^{\dagger}$ denotes the Moore--Penrose pseudoinverse of a matrix.
To achieve high accuracy, the standard approach is to increase the number of quadrature nodes, which substantially raises the computational cost because more matrix factorizations, such as (sparse) LU factorizations, are required.
One could also try to refine accuracy by replacing the generic probing matrix $X$ with the newly computed approximate eigenvectors, like in subspace iteration. However, such a fictitious iterative refinement fails to converge in general. As an illustrative example, consider the \texttt{butterfly} problem from the NLEVP collection~\cite{BHM2013} of dimension \(n=64\). We aim at computing the \(11\) eigenvalues inside the circular domain centered at \(0.35 + 0.25\cdot\mi\) with radius \(0.1\), as shown in the left panel of \cref{fig:NLFEAST}.
When applying Beyn's method with this fictitious iterative refinement, the numerical results presented in the middle panel of \cref{fig:NLFEAST} show that the accuracy of Beyn's method can only be improved by increasing the number of quadrature nodes, but not via iterative refinement.

\begin{figure}[htbp]
\centering
\begin{tabular}{ccc}
\includegraphics[width=0.3\linewidth]{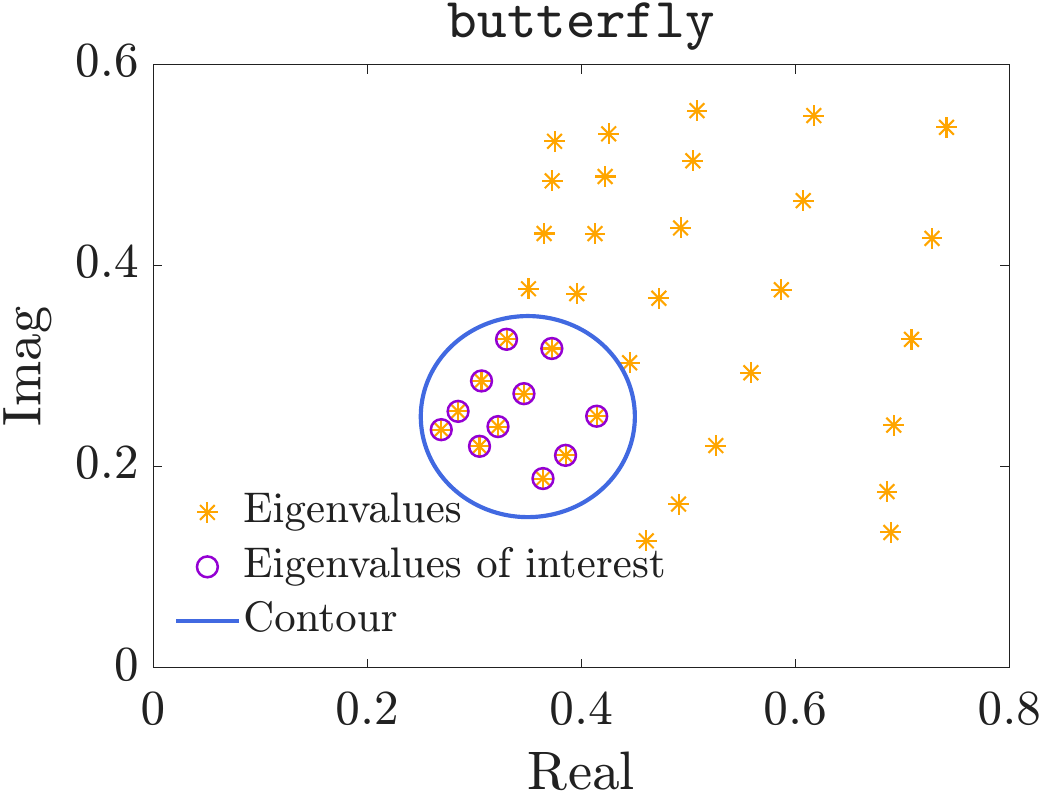}&
\includegraphics[width=0.3\linewidth]{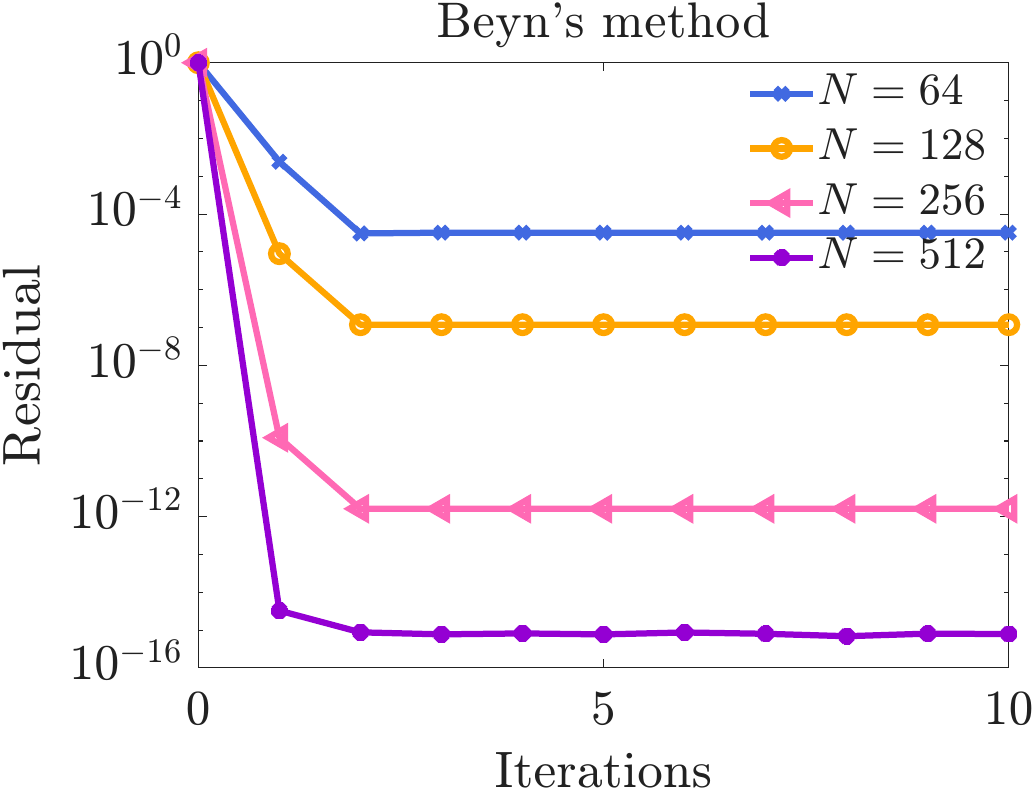}&
\includegraphics[width=0.3\linewidth]{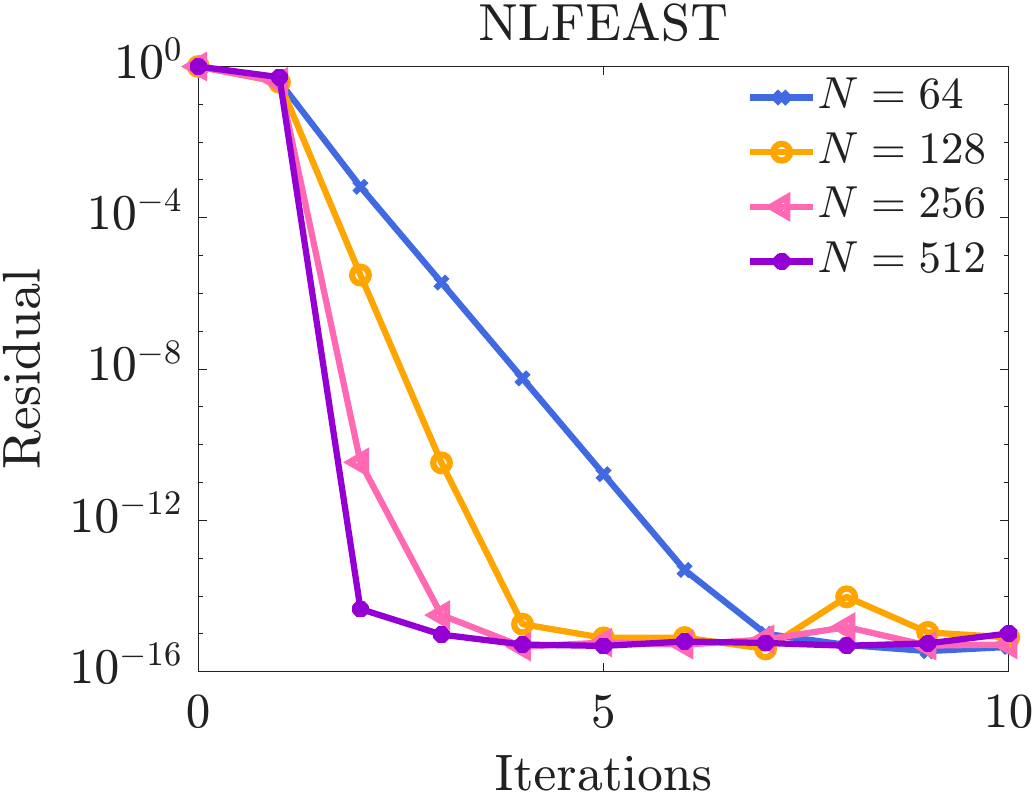}
\end{tabular}
\caption{Beyn's method with (fictitious) iterative refinement and NLFEAST for solving \texttt{butterfly} with dimension \(64\) and different number of quadrature nodes.
The left panel illustrates the domain of interest, which is the circular region centered at \(0.35 + 0.25\cdot\mi\) with radius \(0.1\).
The middle and right panels are convergence histories of iterative refinement applied to Beyn's method and NLFEAST, respectively.
}
\label{fig:NLFEAST}
\end{figure}

For standard eigenvalue problems, it is well known how to turn contour integral-based eigensolvers into iterative methods. For instance, FEAST~\cite{Polizzi2009} is mathematically equivalent to subspace iteration with $r(A)$, where $r$ is a rational approximation of $\chi_{\region}$, typically obtained from applying numerical quadrature to the contour integration representation. 
For NEPs, the nonlinear FEAST method (NLFEAST)~\cite{GMP2018} and its variant NLFEAST-Beyn~\cite{BP2020} have been proposed by combining this rational approximation with residual inverse iteration (RESINVIT)~\cite{Neumaier1985} (see \cref{sec:preliminary} for details). As illustrated in the right panel of \cref{fig:NLFEAST}, NLFEAST behaves empirically as an iterative method for NEPs, in the sense that its accuracy improves progressively with successive iterations.
Although NLFEAST has been widely adopted in scientific computing~\cite{SMB2021,LE2025,GP2021}, its convergence properties remain insufficiently understood from a theoretical standpoint. In contrast to the linear FEAST analysis~\cite{TP2014}, the convergence mechanism of NLFEAST is more intricate and cannot be interpreted simply as a subspace iteration, since the eigenvectors of $T(\xi)$ depend on $\xi$.
In this paper, we propose a framework for solving NEPs with iterative methods and establish local linear convergence under mild assumptions. In particular, our analysis clarifies the convergence behavior of NLFEAST.

The rest of this paper is organized as follows. In \cref{sec:preliminary}, we provide the necessary mathematical preliminaries, starting with a review of the RESINVIT algorithm and Keldysh's theorem. We then introduce a continuous operator for NEPs and characterize its core spectral recovery properties. In \cref{sec:conv}, we present our general framework for iterative nonlinear eigensolvers (\cref{algo:iree}) based on alternating extraction and filtering steps. By establishing a necessary condition for its convergence, we explain why certain contour integral-based nonlinear eigensolvers, such as Beyn's method, fail to improve their accuracy via iterative refinement. We then derive our main local linear convergence result in \cref{thm:main}, and apply this theory to establish the convergence of the NLFEAST algorithm in \cref{lem:nl}. Finally, numerical experiments are presented in \cref{sec:numexp} to validate our theoretical findings on challenging NEPs.

\paragraph{Notation.}
Throughout the paper, $\|\cdot\|$ denotes the Euclidean norm of a vector or the spectral norm of a matrix; $e_{i}$ denotes the $i$th unit vector in $\C^{n}$. For two subspaces $\mathcal{X}_{1}$ and $\mathcal{X}_{2}$ with bases $X_{1}$ and $X_{2}$, respectively, we use $\angle(\mathcal{X}_{1},\mathcal{X}_{2})$ and $\angle(X_{1},X_{2})$ interchangeably to denote the largest principal angle between them \cite[Sec.~6.4.3]{GV2013}. For a sequence of vectors $x_{1},\dotsc,x_{k}$, the notation $\Span\{x_{1},\dotsc,x_{k}\}$ stands for the range of $[x_{1},\dotsc,x_{k}]$. For a matrix $A$, we denote the set of its eigenvalues by $\spec(A)$. 
Throughout this work, the notation $\order(\cdot)$ absorbs constants depending on $T(\cdot)\colon \region\mapsto \C^{n\times n}$ only.

\section{RESINVIT and contour integrals}
\label{sec:preliminary}

As highlighted in~\cite{GMP2018}, NLFEAST can be viewed as an extension of RESINVIT.
Given a shift $\sigma\in\C$, a fixed vector $b\in\C^{n}$ and an initial $(\rho^{(0)},x^{(0)})$, RESINVIT is defined\footnote{While the expression for RESINVIT appears to require $\sigma \not= \rho^{(k)}$, it admits the continuous extension $x^{(k+1)} = T(\sigma)^{-1}T'(\sigma)x^{(k)}$ at $\sigma = \rho^{(k)}$. Here and in the analogous situation for NLFEAST, we tacitly assume that a continuous extension is used in such exceptional cases.   } 
\begin{align*}
    x^{(k+1)}
    &=  \frac{1}{\sigma-\rho^{(k)}}\Bigl(I - T(\sigma)^{-1}T(\rho^{(k)})\Bigr)x^{(k)},\\
    \rho^{(k+1)}
    &= \argmin_{\rho\in\C}\,\bigabs{b^{\Htran}T(\sigma)^{-1}T(\rho)x^{(k+1)}}
\end{align*}
for $k=0$, $1$, $\dotsc$.
If the shift $\sigma$ is chosen sufficiently close to a simple eigenvalue~$\lambda$ of $T(\cdot)$ with an associated eigenvector $u$, then linear convergence of approximate eigenpairs $(\rho^{(k)},x^{(k)})$ produced by RESINVIT to $(\lambda,u)$ is established in \cite{Neumaier1985}. 

When applied to the standard eigenvalue problem $T(\xi)=\xi I-A$, RESINVIT reduces to the shift-and-invert iteration
\begin{equation*}
    x^{(k+1)} = (\sigma I-A)^{-1}x^{(k)},
\end{equation*}
which is a robust algorithm to compute the eigenvalue of $A$ closest to the shift $\sigma$ together with its corresponding eigenvector. The shift-and-invert iteration can be extended to target all eigenvalues of $A$ contained in a prescribed domain $\region$ by iterating over a block of vectors and replacing 
$(\sigma I-A)^{-1}$ with rational approximation to the indicator function $\chi_\region(A)$. Exploiting the relation~\cref{eq:indA} to contour integrals, this leads to the FEAST method mentioned before.

For NEPs, RESINVIT can likewise be combined with a contour integral to compute all eigenvalues contained in a prescribed domain. We begin by formulating this method in a continuous setting. Define
\begin{equation}
    \label{eq:defcif}
\cif(\rho) \defi \frac{1}{2\pi\mi }\oint_{\partial \region}\frac{1}{\xi-\rho}\bigl(I-T(\xi)^{-1}T(\rho)\bigr)\md \xi.    
\end{equation}
For the standard eigenvalue problem $T(\xi)=\xi I-A$, $\cif$ is exactly the spectral projector $\chi_{\region}(A)$, where $\chi_{\region}(\cdot)$ is the indicator function of $\region$.
In order to understand \cref{eq:defcif} for NEPs, we first import the famous Keldysh's theorem \cite{Keldysh1971} characterizing the resolvent of $T(\cdot)$.
\begin{theorem}[Keldysh]
    \label{thm:Keldysh}
    Suppose that a holomorphic matrix-valued function $T(\cdot)$ has~$m$ eigenvalues in $\region$, counted with algebraic multiplicities. Then there exist $n\times m$ matrices $U_{\region}$ and $W_{\region}$, and a matrix $J_{\region}\in\C^{m\times m}$ in Jordan canonical form, such that
    \begin{equation*}
        T(\xi)^{-1}=U_{\region}(\xi I-J_{\region})^{-1}W_{\region}^{\Htran}+H(\xi)\quad\text{for all}\quad \xi\in\region\setminus\spec(J_{\region}),
    \end{equation*}
    where $H(\cdot)$ is a holomorphic matrix-valued function in $\region$. The Jordan structure of $J_{\region}$ reflects the multiplicities of the eigenvalues contained in $\region$.
\end{theorem}
Keldysh's theorem stands as a theoretical foundation for contour integral-based nonlinear eigensolvers, because the holomorphic function $H$ can be annihilated by contour integration.
The following lemma evaluates the contour integral in \cref{eq:defcif} for $\rho\notin \spec(J_{\region})$.
\begin{theorem}
    \label{thm:cif}
    Consider the matrix-valued function $T(\cdot)$ in \cref{thm:Keldysh}.
    For any $\rho \in \region\setminus \spec(J_{\region})$, the function $\cif(\cdot)$ defined in \cref{eq:defcif} satisfies 
    \begin{equation*}
        \cif(\rho) = I-H(\rho)T(\rho) = U_{\region}(\rho I- J_{\region})^{-1}W_{\region}^{\Htran}T(\rho).
    \end{equation*}
\end{theorem}
\begin{proof}
    According to Keldysh's theorem and the Cauchy integration formula, we know that  
    \begin{equation*}
        \cif(\rho) 
        = I-U_{\region}\Bigl(\frac{1}{2\pi\mi }\oint_{\partial \region}\frac{1}{\xi-\rho}(\xi I-J_{\region})^{-1}\md \xi\Bigr) W_{\region}^{\Htran}T(\rho)-H(\rho)T(\rho).
    \end{equation*}
    Note that the inverse of an $\ell$-by-$\ell$ Jordan block is  
    \begin{equation*}
        \begin{bmatrix}
            \xi-\lambda & -1 & &\\ 
            &\xi-\lambda &\ddots &\\ 
            &&\ddots&-1\\ 
            &&&\xi-\lambda
        \end{bmatrix}^{-1}
        =\begin{bmatrix}
            (\xi-\lambda)^{-1} & (\xi-\lambda)^{-2} &\cdots &(\xi-\lambda)^{-\ell}\\ 
            &(\xi-\lambda)^{-1} &\ddots &\vdots\\ 
            &&\ddots&(\xi-\lambda)^{-2}\\ 
            &&&(\xi-\lambda)^{-1}
        \end{bmatrix},
    \end{equation*}
    and the nonzero entries of $(\xi I-J_{\region})^{-1}/(\xi-\rho)$ take the form $\phi_{k}(\xi)=(\xi-\lambda)^{-k}(\xi-\rho)^{-1}$ for some $k\geq 1$. Since $\abs{\phi_{k}(\xi)}=\order(\abs{\xi}^{-2})$ as $\abs{\xi}\to\infty$, the residue theorem asserts that 
    \begin{equation*}
        \frac{1}{2\pi\mi }\oint_{\partial \region}\phi_{k}(\xi)\md \xi = 0
        \quad\text{and}\quad \frac{1}{2\pi\mi }\oint_{\partial \region}\frac{1}{\xi-\rho}(\xi I-J_{\region})^{-1}\md \xi = 0,
    \end{equation*}
    which implies that $\cif(\rho)=I-H(\rho)T(\rho)$. 
Then the proof is completed by using Keldysh's theorem again.
\end{proof}

Let $\lambda$ be an eigenvalue of $T(\cdot)$ in $\region$. Denote the canonical system of generalized eigenvectors of $T(\cdot)$ at $\lambda$ by $U_{\lambda}$, that is, the columns of $U_{\region}$ corresponding to diagonal entries $\lambda$ in $J_{\region}$.
Then \cref{thm:cif} yields that, for any $u_{\lambda}\in \range(U_{\lambda})$, it holds that 
\begin{equation}
  \label{eq:NLproj}
    \lim_{\rho\to\lambda} \cif(\rho)u_{\lambda}\in\range(U_{\lambda}).
\end{equation}

\section{Convergence of iterative nonlinear eigensolvers}

\label{sec:conv}
Although the contour integration can recover $\mathcal{U}_{\region}$ exactly, replacing $\cif(\rho)$ with its discretized counterpart $\itf(\rho)$ in general does \emph{not} yield an iterative method. In this section, we present a framework for a class of  iterative nonlinear eigensolvers based on contour integral and establish a corresponding convergence theory under mild assumptions.

\subsection{A framework for a class of iterative nonlinear eigensolvers}
This framework contains two ingredients in each iteration. 
The first step is an extraction. 
Given a low-dimensional subspace $\mathcal{X}\subset \C^{n}$, we first extract approximate eigenpairs $(\rho_{i},x_{i})$ of $T(\cdot)$ in $\region$, where $x_{i}\in\mathcal{X}$ for $i=1,\dotsc,m$. We hope to obtain approximations $(\rho_{i},x_{i})$ to all exact eigenpairs $(\lambda_{i},u_{i})$ of $T(\cdot)$ in $\region$.
The second step is filtering. We apply a filter $\itf(\rho_{i})$ to each $x_{i}$ as  
\begin{equation*}
    \widehat{x}_{i} = \itf(\rho_{i})x_{i}\quad\text{for}\quad 1\leq i\leq m.
\end{equation*}
Here $\itf(\rho_{i})$ is an approximation of $\cif(\rho_{i})$, obtained by, e.g.,  replacing the contour integral in $\cif(\rho_{i})$, defined in \cref{eq:defcif}, with a quadrature rule.
The whole procedure is summarized in \cref{algo:iree}.

\begin{algorithm}[htbp]
\caption{Iterative nonlinear eigensolvers}
\label{algo:iree}
\begin{algorithmic}[1]
\REQUIRE Matrix-valued functions $T(\cdot)$ and  $\itf$, an initial subspace $\mathcal{X}^{(0)}\subset \C^{n}$.
\ENSURE Approximate eigenpairs $(\rho_{i},x_{i})$ for $1\leq i\leq m$.
\FOR{\(k=0\), \(1\), \(\dotsc\)}
\STATE Extract approximate eigenpairs $(\rho_{i}^{(k)},x_{i}^{(k)})$ for $1\leq i\leq m$ from $\mathcal{X}^{(k)}$.
\COMMENT{Extraction step}
\label{alg-step:proj}
\FOR{\(i=1\), \(2\), \(\dotsc\), \(m\)}
\STATE Compute \(\widehat{x}_{i}^{(k+1)} = \itf(\rho_{i}^{(k)})x_{i}^{(k)}\). \COMMENT{Filtering step}
\label{alg-step:filter}
\ENDFOR
\STATE Form a new subspace $\mathcal{X}^{(k+1)}=\Span\{\widehat{x}_{1}^{(k+1)},\dotsc,\widehat{x}_{m}^{(k+1)}\}$.
\ENDFOR
\end{algorithmic}
\end{algorithm}

\cref{algo:iree} includes the original NLFEAST algorithm proposed in \cite[Algo.~1]{GMP2018} as a special case.
Specifically, let 
\begin{equation*}
    r(\xi) = \sum_{j=1}^{N}\frac{\omega_{j}}{\xi_{j}-\xi}
\end{equation*}
be a rational approximation of the indicator function $\chi_{\region}$ on $\region$. Taking
\begin{equation}
    \label{eq:defnlf}
    \itf(\rho)\equiv \nlf(\rho) \defi \sum_{j=1}^{N} \frac{\omega_{j}}{\xi_{j}-\rho}\Bigl(I-T(\xi_{j})^{-1}T(\rho)\Bigr)
\end{equation} 
in \cref{algo:iree}, and extracting approximate eigenpairs $(\rho_{i}^{(k)},x_{i}^{(k)})$ by a (nonlinear) Rayleigh--Ritz procedure, we obtain NLFEAST summarized in \cref{algo:NLFEAST}.
Here, the Rayleigh--Ritz procedure considers the compressed NEP $(X^{(k)})^{\Htran}T(\cdot)X^{(k)}$ for an orthonormal basis $X^{(k)}$ of $\mathcal{X}^{(k)}$. The Ritz values $\rho_{i}^{(k)}$ are obtained as the eigenvalues of this compressed NEP in $\region$, and the corresponding Ritz vectors are obtained from multiplying $X^{(k)}$ with the associated eigenvectors.

\begin{algorithm}[htbp]
\caption{NLFEAST \cite[Algo.~1]{GMP2018}}
\label{algo:NLFEAST}
\begin{algorithmic}[1]
\REQUIRE A matrix-valued function $T(\cdot)$, and an initial subspace $\mathcal{X}^{(0)}\subset \C^{n}$.
\ENSURE Ritz pairs $(\rho_{i},x_{i})$ for $1\leq i\leq m$.
\STATE Compute quadrature nodes and weights $(\xi_{j},\omega_{j})$ for $1\leq j\leq N$.
\FOR{\(k=0\), \(1\), \(\dotsc\)}
\STATE Extract Ritz pairs $(\rho_{i}^{(k)},x_{i}^{(k)})$ for $1\leq i\leq m$ in $\mathcal{X}^{(k)}$ by the Rayleigh--Ritz procedure.
\label{alg-step:proj_NLFEAST}

\FOR{\(i=1\), \(2\), \(\dotsc\), \(m\)}
\STATE Compute 
\[
\widehat{x}_{i}^{(k+1)} =
\sum_{j=1}^{N} \frac{\omega_{j}}{\xi_{j}-\rho_{i}^{(k)}}\Bigl(I-T(\xi_{j})^{-1}T(\rho_{i}^{(k)})\Bigr)x_{i}^{(k)}.
\]
\label{alg-step:filter_NLFEAST}
\ENDFOR
\STATE Form a new subspace $\mathcal{X}^{(k+1)}=\Span\{\widehat{x}_{1}^{(k+1)},\dotsc,\widehat{x}_{m}^{(k+1)}\}$.
\ENDFOR
\end{algorithmic}
\end{algorithm}

In \cref{sec:nlfeast_mod}, we also introduce some variants of NLFEAST that fit within the framework of \cref{algo:iree} and present some numerical results to illustrate their convergence.

\subsection{A necessary condition for convergence}
\label{sec:nece}

\cref{algo:iree} can be viewed as an iterative method on the subspace $\mathcal{X}^{(k)}$. Denote 
\begin{equation*}
    \mathcal{X}^{(k+1)} = \itfo(\mathcal{X}^{(k)}).
\end{equation*}
Then \cref{algo:iree} converges locally if $\itfo$ is a contraction mapping in a neighborhood of the fixed point $\mathcal{U}_{\region} \defi \range(U_{\region})$ on the Grassmann manifold.
Clearly, a necessary condition for $\mathcal{U}_{\region}$ being a fixed point is that, for any eigenvalue $\lambda$ of $T(\cdot)$ in $\region$, it holds that 
\begin{equation}
    \label{eq:neceCon}
    \lim_{\rho\to\lambda}\itf(\rho)u \in\mathcal{U}_{\region}
    \quad\text{for}\quad u\in\range(U_{\lambda}),
\end{equation}
where $U_{\lambda}$ is the canonical system of generalized eigenvectors of $T(\cdot)$ at $\lambda$ as in \cref{eq:NLproj}.
The condition \cref{eq:neceCon} is essentially a weaker variant of the continuous case in \cref{eq:NLproj}.

The necessary condition~\cref{eq:neceCon} explains the difficulty of converting certain nonlinear eigensolvers into iterative methods. We start with Beyn's method. Let
\begin{equation*}
    M_{1,N} = \sum_{j=1}^{N}\omega_{j}\xi_{j}T(\xi_{j})^{-1}
    \quad\text{and}\quad 
    M_{0,N} = \sum_{j=1}^{N}\omega_{j}T(\xi_{j})^{-1},
\end{equation*}
where the weights $\omega_{j} \in \C$ and nodes $\xi_{j}$ are chosen such that $M_{s,N} \approx M_{s}$, with $M_{s}$ defined in \cref{eq:defmon}.
As already discussed in \cref{sec:intro}, one could try to refine Beyn's method like in subspace iteration.
The fixed point of this fictitious iterative refinement corresponds to an invariant subspace of the matrix pencil $M_{1,N}-\xi M_{0,N}$, which varies with $N$. In contrast, for a generic NEP, the desired subspace $\mathcal{U}_{\region}$ is, in general, not an invariant subspace of $M_{1,N}-\xi M_{0,N}$ unless the numerical quadrature incurs no error.
Consequently, quadrature introduces an error into Beyn's method, as well as the (block) SS method for similar reasons, that cannot be mitigated by iterative refinement.

The next method we consider is a kind of discretization of $\cif(\rho)$ in \cref{eq:defcif} proposed in \cite[Eq.~(13)]{GMP2018}, which is already noted there to \emph{not} work: 
\begin{equation*}
    \widehat{x}_{i}^{(k+1)} = \sum_{j=1}^{N}\omega_{j}T(\xi_{j})^{-1}T^{\prime}(\rho_{i}^{(k)})x_{i}^{(k)} =: \itf( \rho_i^{(k)}) x_{i}^{(k)}.
\end{equation*}
To understand this from a theoretical point of view, we look into the asymptotic behavior of one simple eigenpair $(\lambda,u)$ as follows:
\begin{equation*}
\lim_{\rho\to\lambda}  \itf(\rho) u =    \lim_{\rho\to\lambda} \sum_{j=1}^{N}\omega_{j}T(\xi_{j})^{-1}T^{\prime}(\rho)u = \sum_{j=1}^{N}\omega_{j}T(\xi_{j})^{-1}T^{\prime}(\lambda)u=: F_{N}T^{\prime}(\lambda)u.
\end{equation*}
Note that $T^{\prime}(\lambda)u\in \C^{n}$ is typically a nonzero vector independent of $N$.
For a generic NEP, the matrix $F_{N}$ varies with $N$, making it unlikely that $F_{N}T^{\prime}(\lambda)u\in\mathcal{U}_{\region}$ holds for some prescribed~$N$.
Hence, $\mathcal{U}_{\region}$ is unlikely to be a fixed point of $\itfo$, which implies that the scheme above is not an iterative method.

\subsection{A sufficient condition for linear convergence}
\label{sub-sec:suffcond}

In this section, we derive one of the main results of this work,~\cref{thm:main}, which establishes local linear convergence provided that a sufficiently good filter and extraction mechanism is used. For this purpose,  we will first introduce~\cref{asp:RR} on extraction and link it to a suitable notion of eigenpair accuracy. As our result will be asymptotic with respect to filter quality, we will need to consider a whole filter family that contains increasingly good filters.

Throughout this section, we make the following assumption on $T(\cdot)$ for the sake of simplicity in our analysis. 
\begin{assumption}
    \label{asp:T}
Assume that the eigenvalues $\lambda_{1},\dotsc,\lambda_{m}$ in $\region$ are simple, and their corresponding (unit) eigenvectors $u_{1},\dotsc,u_{m}$ are linearly independent.
\end{assumption}

Under \cref{asp:T}, we know that the number of eigenvalues in $\region$ satisfies $m\leq n$. 
Moreover, since all eigenvalues are simple, then the eigenvectors of $T(\cdot)$ are unique up to scaling. If we choose the right and left eigenvectors in accordance with 
Keldysh's theorem (that is, the eigenvectors are columns of the matrices $U_\region$ and $W_\region$ in~\cref{thm:Keldysh}) then the residue theorem yields that  
\begin{equation*}
    u_{i}w_{i}^{\Htran} = \frac{1}{2\pi\mi}\oint_{\mathcal C_{i}}T(\xi)^{-1}\md \xi = \frac{1}{w_{i}^{\Htran}T^{\prime}(\lambda_{i})u_{i}}u_{i}w_{i}^{\Htran} \quad\Rightarrow \quad w_{i}^{\Htran}T^{\prime}(\lambda_{i})u_{i}=1,
\end{equation*}
where $\mathcal C_{i}$ is a small contour that only encloses one eigenvalue $\lambda_{i}$.
\begin{remark}

\Cref{asp:T} excludes certain challenging scenarios. Nevertheless, with a careful implementation, \cref{algo:iree} can still be applied and achieves linear convergence even when \cref{asp:T} is violated; see \cref{subsec:numexpSameEvec} for a numerical demonstration.
From a theoretical standpoint, \cref{asp:T} could be relaxed by invoking the concept of invariant pairs \cite{Kressner2009}.
\end{remark}

\subsubsection{Assumption on eigenpair extraction}
To establish a convergence result, we also need the following assumption on the quality of approximate eigenpairs extracted in line~\ref{alg-step:proj} of \cref{algo:iree}.
\begin{assumption}
    \label{asp:RR}
    Given  \cref{asp:T}, we assume that there exists a sufficiently small constant $0<\epsilon_{T}<\pi/2$ depending only on $T(\cdot) \colon \region \mapsto \C^{n\times n}$, such that the following holds. For any $m$-dimensional subspace $\mathcal{X}\subset\C^{n}$ satisfying $\epsilon\defi \angle(\mathcal{X},\mathcal{U}_{\region})\leq \epsilon_{T}$ for $\mathcal{U}_{\region}=\Span\{u_{1},\dotsc,u_{m}\}$, the eigenpairs $(\rho_{i},x_{i})$ extracted from $\mathcal{X}$ in line~\ref{alg-step:proj} of \cref{algo:iree} are $\order(\epsilon)$-approximations  to the (appropriately ordered) exact eigenpairs $(\lambda_{i},u_{i})$ of $T(\cdot)$ in $\region$:
    \begin{equation*}
        \abs{\rho_{i}-\lambda_{i}}=  \order(\epsilon)\quad\text{and}\quad \angle(x_{i},u_{i})=   \order(\epsilon), \quad i = 1,\ldots, m.
    \end{equation*}
Here, the constants in $\order(\cdot)$ are independent of the choice of $\mathcal{X}$.
\end{assumption} 

\cref{asp:RR} essentially says that, once the subspace $\mathcal{X}^{(k)}$ is a good approximation to $\mathcal{U}_{\region}$, then we can extract equally good approximate eigenpairs from it. In particular, for the exact $\mathcal{U}_{\region}$, the eigenvalues and eigenvectors are recovered exactly by the extraction. 

\begin{remark}
For standard eigenvalue problems, the standard Rayleigh--Ritz procedure
satisfies \cref{asp:RR}.
For NEPs, the nonlinear Rayleigh--Ritz procedure used in \cref{algo:NLFEAST} is \emph{usually} observed to satisfy \cref{asp:RR}, but it may fail for certain $T(\cdot)$'s because the compressed NEP may be $\epsilon$-close to singular; see \cite[Ex.~3]{shao2026stabilizing} for example.
The \emph{randomized Rayleigh--Ritz} procedure recently introduced in \cite{shao2026stabilizing} satisfies the convergence guarantee required in \cref{asp:RR} with high probability.
\end{remark}

\subsubsection{Decomposition of approximate eigenvectors}
Recalling $U_{\region}=[u_{1},\dotsc,u_{m}]$, we denote $U_{\out}$ as an orthonormal basis of the orthogonal complement of $\mathcal{U}_{\region} = \range(U_{\region})$ and $U=[U_{\region},U_{\out}]\in\C^{n\times n}$.
Let us consider an arbitrary $m$-dimensional subspace $\mathcal{X}\subset \C^{n}$ in a sufficiently small neighborhood of $\mathcal{U}_{\region}$ on the Grassmann manifold. Let $\{(\rho_{i},x_{i})\}_{i=1}^{m}$ be the approximate eigenpairs of $T(\cdot)$ extracted from $\mathcal{X}$ in line~\ref{alg-step:proj} of \cref{algo:iree}. Without loss of generality, we assume that~$x_{i}$ is normalized by $e_{i}^{\Htran}U^{-1}x_{i}=1$, and hence, it admits the following decomposition: 
\begin{equation}
    \label{eq:decompX}
    x_{i} = u_{i}+U_{-i}y_{i}+U_{\out}z_{i}\quad\text{for}\quad i=1,\dotsc,m,
\end{equation}
where $U_{-i}=[u_{1},\dotsc,u_{i-1},u_{i+1},\dotsc,u_{m}]$.
Even though $U$ is not an orthogonal matrix, it is invertible and fixed with respect to $i$; hence, $\norm{y_{i}}$ and $\norm{z_{i}}$ are still indicative of how closely $x_{i}$ approximates $u_{i}$.
The following proposition shows that $\angle(\mathcal{X},\mathcal{U}_{\region})$ is a good measurement for the accuracy of all approximate eigenpairs. 
\begin{proposition}
\label{prop:eps}
    Let $\epsilon\defi \angle(\mathcal{X},\mathcal{U}_{\region})\leq \epsilon_{T}$. Consider the extracted eigenpairs $(\rho_{i},x_{i})$ of $T(\cdot)$ from the subspace $\mathcal{X}$, where $x_{i}$ admits the decomposition \cref{eq:decompX}.
    Under \cref{asp:RR}, it holds that 
    \begin{equation*}
        \max_{1\leq i\leq m} \max\Bigl\{\abs{\rho_{i}-\lambda_{i}},\norm{y_{i}},\norm{z_{i}}\Bigr\} = \order(\epsilon).  
    \end{equation*} 
\end{proposition}
\begin{proof}
    The property $\abs{\rho_{i}-\lambda_{i}}=\order(\epsilon)$ follows directly from the convergence of Ritz values in \cref{asp:RR}. The remaining terms follow from the convergence of Ritz vectors:
    \begin{equation*}
        \norm{z_{i}}^{2}+\norm{y_{i}}^{2} = \tan^{2}\angle(U^{-1}x_{i},e_{i}) = \order\bigl( \tan^{2}\angle(x_{i},u_{i})\bigr)  =\order(\epsilon^{2}),
    \end{equation*} 
    where we used that $U$ is invertible and a constant with respect to $\epsilon$.
\end{proof}

\subsubsection{Effectiveness of a filter family} \label{sec:filterfamily}

We now consider a family of filters $\itf\equiv\itf^{[N]}$ parametrized by a positive integer $N$. 
Let $\widehat{x}_{i}\defi\itf^{[N]}(\rho_{i})x_{i}$, and consider a decomposition analogous to~\eqref{eq:decompX}:
\begin{equation}
        \label{eq:decompXh}
    \widehat{x}_{i} = \widehat{\alpha}_{i}u_{i}+U_{-i}\widehat{y}_{i}+U_{\out}\widehat{z}_{i}.
\end{equation}    
While the variables $\widehat{\alpha}_{i}$, $\widehat{y}_{i}$ and $\widehat{z}_{i}$ of course depend on $N$, we omit the superscript $\cdot^{[N]}$ for notational clarity.

We require the filter family to be effective in the sense that components in $\mathcal{U}_{\out}$ are reduced by a factor in every step, while components in $\mathcal{U}_{\region}$ are nearly preserved.
Specifically, we assume that 
\begin{equation}
    \label{eq:aspfilter}
    \bigabs{\widehat{\alpha}_{i}-1} = \order(\gamma_{N}+\epsilon),\quad \norm{\widehat{y}_{i}}= \order(\gamma_{N}+\epsilon)\quad\text{and}\quad  \norm{\widehat{z}_{i}}= \order(\gamma_{N}\epsilon), \quad i = 1,\ldots,m,
\end{equation}
where $\epsilon = \angle(\mathcal{X},\mathcal{U}_{\region})$ as before. 
The quantity $\gamma_N \ge 0$ quantifies the quality of the filter family and we assume that $\gamma_{N} \stackrel{N\to \infty}{\longrightarrow} 0$.
In~\cref{sub-sec:NLFEAST}, we show that the filter~\cref{eq:defnlf} used in NLFEAST defines such a family; it satisfies \cref{eq:aspfilter} with the parameter $N$ representing the number of quadrature nodes.

\subsubsection{Linear convergence of \cref{algo:iree}}

We are now ready to formulate our first main result on the convergence of \cref{algo:iree}.
\begin{theorem}
    \label{thm:main}
    Let $T(\cdot)$ be an NEP satisfying \cref{asp:T} for a domain $\region$. Using the notation of \cref{asp:RR}, consider an $m$-dimensional subspace  $\mathcal{X}$ such that $\epsilon\defi \angle(\mathcal{X},\mathcal{U}_{\region})\leq \epsilon_{T}$, and let $(\rho_{i},x_{i})$, $i = 1,\ldots,m$, denote approximate eigenpairs of $T(\cdot)$ extracted from $\mathcal{X}$ satisfying \cref{asp:RR}. 
    Assume that there exists a family of filters $\itf^{[N]}(\cdot)$ such that $\widehat{x}_{i}\defi\itf^{[N]}(\rho_{i})x_{i}$ satisfies \cref{eq:aspfilter} with $\gamma_{N} \stackrel{N\to \infty}{\longrightarrow} 0$.
    Then the updated subspace $\widehat{\mathcal{X}} = \Span\{\widehat{x}_{1},\dotsc,\widehat{x}_{m}\}$ satisfies
    \begin{equation*}
        \widehat{\epsilon}\defi \angle(\widehat{\mathcal{X}},\mathcal{U}_{\region})=\order(\gamma_{N}\epsilon).
    \end{equation*}
    Here, the constants in $\order(\cdot)$ are independent of the choice of $\mathcal{X}$.
\end{theorem}
\cref{thm:main} shows---under the stated assumptions---that one step of \cref{algo:iree} effects a linear contraction of $\angle(\mathcal{X},\mathcal{U}_{\region})$.
By applying \cref{thm:main} recursively, it follows that, for sufficiently large $N$, \cref{algo:iree} exhibits local linear convergence with a rate proportional to $\gamma_{N}$.

\begin{proof}[Proof of \cref{thm:main}]
    Let $X=[x_{1},\dotsc,x_{m}]$ and  $\widehat{X}=[\widehat{x}_{1},\dotsc,\widehat{x}_{m}]$. Since 
    \begin{equation*}
        \widehat{\epsilon} = \angle(\widehat{X},U_{\region}) = \order\bigl(\angle(U^{-1}\widehat{X},E_{m})\bigr)=\order\bigl(\tan\angle(U^{-1}\widehat{X},E_{m})\bigr),
    \end{equation*}
    where $E_{m}=[e_{1},\dotsc,e_{m}]$.
    To establish a linear convergence, it suffices to show 
    \begin{equation}
        \label{eq:tanleft}
        \tan\angle\bigl(U^{-1}\widehat{X},E_{m}\bigr) = \order(\gamma_{N}\epsilon).
    \end{equation}
    Let $\widehat{Y}= [\widehat{y}_{1},\dotsc,\widehat{y}_{m}]$ and $\widehat{Z}= [\widehat{z}_{1},\dotsc,\widehat{z}_{m}]$.
    According to \cref{eq:aspfilter} and the triangle inequality, we have
    \begin{equation*}
            \norm{U_{\out}^{\dagger}\widehat{X}}=\norm{\widehat{Z}}=\order(\gamma_{N}\epsilon)
            \quad\text{and}\quad 
            \norm{U_{\region}^{\dagger}\widehat{X}-I} \leq  \norm{\widehat{Y}}+\max_{1\leq i\leq m}\abs{\widehat{\alpha}_{i}-1} =  \order(\gamma_{N}+\epsilon).
    \end{equation*}
    As a consequence, we know that $\norm{(U_{\region}^{\dagger}\widehat{X})^{-1}}=\order(1)$.
    Then \cref{eq:tanleft} is obtained by a representation of the tangent of principal angles in \cite[Thm.~3.1]{Zhu2013} and the submultiplicativity of the spectral norm:
    \begin{equation*}
        \tan\angle\bigl(U^{-1}\widehat{X},E_{m}\bigr) 
        =\norm{U_{\out}^{\dagger}\widehat{X}(U_{\region}^{\dagger}\widehat{X})^{-1}} \leq \norm{U_{\out}^{\dagger}\widehat{X}}\norm{(U_{\region}^{\dagger}\widehat{X})^{-1}} = \order(\gamma_{N} \epsilon),
    \end{equation*}
    and in turn, the theorem is proved.
\end{proof}

As we have already seen in \cref{sec:intro} and \cref{sec:nece}, some contour integral discretizations do not combine well with 
the iterative refinement performed by~\cref{algo:iree}.
To gain further theoretical insight, let us consider the convergence to an individual eigenpair, w.l.o.g., the first eigenpair $(\lambda_{1},u_{1})$. For the filter family $\itf^{[N]}$ to satisfy \cref{eq:aspfilter} we essentially need to ensure that
\begin{equation}
     \label{eq:condblablab}
    \itf^{[N]}(\rho_{1}) = U\begin{bmatrix}
        1+\order(\gamma_{N}+\epsilon) & \order(1) & \order(1)\\ 
        \order(\gamma_{N}+\epsilon) & \order(1) & \order(1)\\ 
        \order(\gamma_{N}\epsilon) &\order(\gamma_{N}+\epsilon) &\order(\gamma_{N}+\epsilon)
    \end{bmatrix}U^{-1}
\end{equation}
holds, where the block partitioning is according to \cref{eq:decompX} and we recall from~\cref{prop:eps} that $\abs{\rho_{1}-\lambda_{1}}=\order(\epsilon)$.
The most challenging, but nevertheless important part is to ensure the asymptotics for the $(3,1)$ block.
Given the contour integral-based operator $\cif$, \cref{thm:cif} yields 
\begin{equation*} 
\lim_{\rho_{1}\to\lambda_{1}}\cif(\rho_{1}) = U\begin{bmatrix}
        1 & \order(1) & \order(1)\\ 
        0 & \order(1) & \order(1)\\ 
        0 & 0 & 0 
    \end{bmatrix}U^{-1}. 
\end{equation*}
Both discretizing the contour integration and changing the eigenvalue $\lambda_{1}$ to an approximate eigenvalue $\rho_{1}$ will perturb the  eigenvectors of $\cif(\lambda_{1})$. 
Unless particular attention is paid, as is done in NLFEAST, this can be expected to introduce 
an error proportional to $\gamma_{N}+\abs{\rho_{1}-\lambda_{1}}$ into the $(3,1)$ block of $\itf^{[N]}(\rho_{1})$. As a consequence,
condition~\eqref{eq:condblablab} is not satisfied.
Also, $\lim_{\rho_{1}\to\lambda_{1}}\itf^{[N]}(\rho_{1}) u_{1}$ is unlikely to be contained in $\mathcal{U}_{\region}$
and hence, $\mathcal{U}_{\region}$ is \emph{not} the fixed point of the operator $\lim_{\rho_{1}\to\lambda_{1}}\itfo^{[N]}(\rho_{1})$.
Thus, some contour integral-based nonlinear iterative eigensolvers fail to improve their accuracy via performing iterations.

\subsection{Linear convergence of NLFEAST}
\label{sub-sec:NLFEAST}

In this section, we aim at showing that the filter family $\nlf$, defined in~\cref{eq:defnlf} and used in NLFEAST, is effective in the sense of \cref{eq:aspfilter}.
Since $\cif(\rho)$ is a bounded operator and the integrand is holomorphic on some domain that contains the closure of $\region$, we can assume that our numerical quadrature converges exponentially \cite[Thm.~19.3]{Trefethen2013}, that is, there exists a $\gamma_{N}= C_T \exp(-c_{T}N)$, where $c_T,C_T>0$ depend only on $T\colon \region\mapsto \C^{n\times n}$, such that
\begin{equation}
    \label{eq:quadrature}
    \max\Bigl\{\Bignorm{\sum_{j=1}^{N}\frac{\omega_{j}}{\xi_{j}-\rho}H(\xi_{j}) - H(\rho)},
    \bignorm{\cif(\rho)-\nlf(\rho)}
    \Bigr\} \leq \gamma_{N} \quad\text{for all}\quad \rho\in \region_{0},
\end{equation}
where $H(\cdot)$ is the holomorphic matrix-valued function in \cref{thm:Keldysh}
and $\region_{0} \subset \region$ is a compact set containing small neighborhoods of the desired eigenvalues $\lambda_1,\ldots,\lambda_m$. In the following theorem, we show that \cref{eq:quadrature} suffices to ensure the effectiveness of the filter family in NLFEAST.
\begin{theorem}
    \label{lem:nl}
    Let $T(\cdot)$ be an NEP satisfying \cref{asp:T}.
    Consider an arbitrary approximate eigenpair $(\rho_{i},x_{i})$ and let $\widehat{x}_{i}\defi \nlf(\rho_{i})x_{i}$, where $x_{i}$ and $\widehat{x}_{i}$ admit the decomposition \cref{eq:decompX,eq:decompXh}, respectively, and $1\leq i\leq m$.
    Assume that the accuracy of quadrature in $\nlf(\cdot)$ satisfies \cref{eq:quadrature}, and 
    \begin{equation*}
        \epsilon \defi \max_{1\leq i\leq m} \max \Bigl\{\abs{\rho_{i}-\lambda_{i}},\norm{y_{i}},\norm{z_{i}}\Bigr\}
    \end{equation*}
    is sufficiently small, then the effectiveness condition \cref{eq:aspfilter} holds for all $(\rho_{i},x_{i})$.
\end{theorem}
\begin{proof}
    Since all eigenvalues of $T(\cdot)$ in $\region$ are simple, $\rho_{i}\neq \lambda_{j}$ for $i\neq j$ whenever $\epsilon$ is sufficiently small. In the following, we will treat two different situations: $\rho_{i}=\lambda_{i}$ or $\rho_{i}$ is not an eigenvalue of $T(\cdot)$.

    Let $[v_{1},\dotsc,v_{n}]=U^{-\Htran}$ and $V_{-i}=[v_{1},\dotsc,v_{i-1},v_{i+1},\dotsc,v_{m}]$. 
    Using \cref{eq:quadrature}, we know that 
    \begin{equation*}
        \abs{\widehat{\alpha}_{i}-1} = \bigabs{v_{i}^{\Htran}\nlf(\rho_{i})x_{i}-1} =  \bigabs{v_{i}^{\Htran}\cif(\rho_{i})x_{i}-1}+\order(\gamma_{N}) = \bigabs{v_{i}^{\Htran}\cif(\rho_{i})u_{i}-1}+\order(\gamma_{N}+\epsilon). 
    \end{equation*}
    When $\rho_{i}=\lambda_{i}$, we have $T(\rho_{i})u_{i}=0$. As a consequence, 
    \begin{equation*}
        v_{i}^{\Htran}\cif(\rho_{i})u_{i} = \frac{1}{2\pi\mi }\oint_{\partial \region}\frac{1}{\xi-\rho_{i}}v_{i}^{\Htran}\bigl(I-T(\xi)^{-1}T(\rho_{i})\bigr)u_{i}\md \xi = \frac{1}{2\pi\mi }\oint_{\partial \region}\frac{v_{i}^{\Htran}u_{i}}{\xi-\rho_{i}}\md \xi = 1.   
    \end{equation*}
    Otherwise, when $\rho_{i} \not=\lambda_{i}$, recall that $J_{\region} = \diag\{\lambda_{1},\dotsc,\lambda_{m}\}$. \Cref{thm:cif} yields that 
    \begin{equation*}
        \cif(\rho_{i}) = U_{\region}(\rho_{i}I-J_{\region})^{-1}W_{\region}^{\Htran}T(\rho_{i}).
    \end{equation*}
    Using the facts that $T(\lambda_{i})u_{i}=0$ and $w_{i}^{\Htran}T^{\prime}(\lambda_{i})u_{i}=1$, we have 
    \begin{align*}
        v_{i}^{\Htran}\cif(\rho_{i})u_{i}
        &= v_{i}^{\Htran}U_{\region}(\rho_{i}I-J_{\region})^{-1}W_{\region}^{\Htran}T(\rho_{i})u_{i} \\
        &= \frac{w_{i}^{\Htran}T(\rho_{i})u_{i}}{\rho_{i}-\lambda_{i}} \\
        &= \frac{w_{i}^{\Htran}\bigl(T(\rho_{i})-T(\lambda_{i})\bigr)u_{i}}{\rho_{i}-\lambda_{i}} \\
        &= w_{i}^{\Htran}T^{\prime}(\lambda_{i})u_{i}+\order(\abs{\rho_{i}-\lambda_{i}}) \\
        &= 1+\order(\epsilon).
    \end{align*}
    Thus, $\abs{\widehat{\alpha}_{i}-1}= \order(\gamma_{N}+\epsilon)$ holds for both situations.

    With a similar calculation, it follows that
    \begin{equation*}
        \norm{\widehat{y}_{i}} = \norm{V_{-i}^{\Htran}\nlf(\rho_{i})x_{i}} =  \norm{V_{-i}^{\Htran}\cif(\rho_{i})u_{i}}+\order(\gamma_{N}+\epsilon) = \order(\gamma_{N}+\epsilon),    
    \end{equation*}
    where, in the last equality, we used the facts that $V_{-i}^{\Htran}u_{i}=0$ and  
    \begin{align*}
        v_{k}^{\Htran}\cif(\rho_{i})u_{i}
        &= v_{k}^{\Htran}U_{\region}(\rho_{i}I-J_{\region})^{-1}W_{\region}^{\Htran}T(\rho_{i})u_{i} \\
        &= \frac{w_{k}^{\Htran}T(\rho_{i})u_{i}}{\rho_{i}-\lambda_{k}} \\
        &= \frac{w_{k}^{\Htran}\bigl(T(\rho_{i})-T(\lambda_{i})\bigr)u_{i}}{\rho_{i}-\lambda_{k}} \\
        &=\order(\epsilon)
    \end{align*}
    for any $1\leq i\neq k\leq m$ and $\rho_{i}$ not being an eigenvalue of $T(\cdot)$ in $\region$.
    When $\rho_{i}=\lambda_{i}$, the desired estimate for $\norm{\widehat{y}_{i}}$ immediately follows from 
    $V_{-i}^{\Htran}\cif(\lambda_{i})u_{i} = V_{-i}^{\Htran}u_{i}=0$.

    It remains to bound $\norm{\widehat{z}_{i}}$. For this purpose, we first note that $U_{\out}^{\dagger}\cif(\rho_{i})=0$ and $\norm{x_{i}-u_{i}}= \order(\epsilon)$ imply
    \begin{align*}
        \norm{\widehat{z}_{i}} &= \norm{U_{\out}^{\dagger}\nlf(\rho_{i})x_{i}} \\
        &\leq
        \norm{U_{\out}^{\dagger}\nlf(\rho_{i})u_{i}}+
        \norm{U_{\out}^{\dagger}\nlf(\rho_{i})(x_{i}-u_{i})} \\ 
        &=
        \norm{U_{\out}^{\dagger}\nlf(\rho_{i})u_{i}}+
        \norm{U_{\out}^{\dagger}\bigl(\nlf(\rho_{i})-\cif(\rho_{i})\bigr)(x_{i}-u_{i})} \\
        &=
        \norm{U_{\out}^{\dagger}\nlf(\rho_{i})u_{i}}+\order(\gamma_{N}\epsilon).
    \end{align*}
    Now it suffices to show that $\norm{U_{\out}^{\dagger}\nlf(\rho_{i})u_{i}}=\order(\gamma_{N}\epsilon)$.
    When $\rho_{i}=\lambda_{i}$, this term vanishes because
    \begin{equation*}
        U_{\out}^{\dagger}\nlf(\lambda_{i})u_{i} = U_{\out}^{\dagger}\sum_{j=1}^{N} \frac{\omega_{j}}{\xi_{j}-\lambda_{i}}\Bigl(I-T(\xi_{j})^{-1}T(\lambda_{i})\Bigr)u_{i} = \sum_{j=1}^{N} \frac{\omega_{j}}{\xi_{j}-\lambda_{i}}\cdot U_{\out}^{\dagger}u_{i} = 0,
    \end{equation*}
    where we used $U_{\out}^{\dagger}u_{i} = 0$ in the last equality.
    Otherwise, by \cref{thm:cif} and $
        U_{\out}^{\dagger}U_{\region}=0$, we know that 
    \begin{equation*}
        U_{\out}^{\dagger}H(\rho_{i})T(\rho_{i})u_{i} = U_{\out}^{\dagger}(I-\cif(\rho_{i}))u_{i}= U_{\out}^{\dagger}\bigl(I-U_{\region}(\rho_{i}I-J_{\region})^{-1}W_{\region}T(\rho_{i})\bigr)u_{i}=0.
    \end{equation*}
    Using Keldysh's theorem, we derive
    \begin{align*}
        U_{\out}^{\dagger}\nlf(\rho_{i})u_{i} &= -U_{\out}^{\dagger}\sum_{j=1}^{N}\frac{\omega_{j}}{\xi_{j}-\rho_{i}}T(\xi_{j})^{-1}T(\rho_{i})u_{i} \\ 
        &= -U_{\out}^{\dagger}\sum_{j=1}^{N}\frac{\omega_{j}}{\xi_{j}-\rho_{i}} \bigl(U_{\region}(\xi_{j}I-J_{\region})^{-1}W_{\region}^{\Htran}+H(\xi_{j})\bigr)T(\rho_{i})u_{i}\\ 
        &= -U_{\out}^{\dagger}\sum_{j=1}^{N}\frac{\omega_{j}}{\xi_{j}-\rho_{i}} H(\xi_{j})T(\rho_{i})u_{i}\\
        &= U_{\out}^{\dagger} \Bigl( H(\rho_i) - \sum_{j=1}^{N}\frac{\omega_{j}}{\xi_{j}-\rho_{i}} H(\xi_{j}) \Bigr) \bigl(T(\rho_{i})-T(\lambda_{i})\bigr)u_{i}\\
        &=\order(\gamma_{N}\epsilon).
    \end{align*}
Here, we used $U_{\out}^{\dagger}U_{\region}=0$ in the third equality,
$U_{\out}^{\dagger} H(\rho_i) T(\rho_i) u_i = 0$ and $T(\lambda_{i})u_{i}=0$ in the fourth equality,
as well as~\cref{eq:quadrature} in the last equality.
\end{proof}

Plugging \cref{lem:nl} into \cref{thm:main} and using \cref{prop:eps}, we obtain the convergence of NLFEAST.

\begin{corollary}
\label{cor:NLFEAST}
Consider NLFEAST with the filter $\nlf(\cdot)$ applied to an NEP $T(\cdot)$ satisfying \cref{asp:T}. 
Under \cref{asp:RR}, there exist constants $\epsnl>0$ and $\cnl>0$ depending only on $T(\cdot)\colon \region\mapsto \C^{n\times n}$ such that the following statements hold.

Assume that the accuracy of the numerical quadrature in \cref{eq:quadrature} satisfies $\gamma_{N}<1/\cnl$, then 
\begin{equation*}
    \angle(\mathcal{X}^{(k)},\mathcal{U}_{\region}) \leq (\cnl\gamma_{N})^{k} \angle(\mathcal{X}^{(0)},\mathcal{U}_{\region})
\end{equation*} 
holds for any initial subspace $\mathcal{X}^{(0)}$ satisfying $\angle(\mathcal{U}_{\region},\mathcal{X}^{(0)})\leq\epsnl$. Here, $\gamma_{N} =C_T \exp(-c_{T}N)$ for some constants $c_{T}$ and $C_T$ depending only on $T\colon \region\mapsto \C^{n\times n}$.
\end{corollary}

\begin{remark}
The constant $\cnl$ in \cref{cor:NLFEAST} may depend on the eigenvalue and eigenvector condition numbers of $T(\cdot)$.
This is different from subspace iteration or FEAST for linear (non-Hermitian) eigenvalue problems since in NLFEAST, we have to perform extraction in each iteration. However, since $\gamma_{N}$ can be made arbitrarily close to zero by using a sufficiently large number of quadrature nodes, NLFEAST can always be made (linearly) convergent.
\end{remark}

To conclude this section, we stress that although the convergence of NLFEAST also requires a sufficiently accurate quadrature rule in \cref{eq:quadrature}, this requirement is fundamentally different from the ``failure'' of being an iterative method. The ``failure'' refers to the fact that, for any inexact quadrature rule, the eigenpairs produced by those methods, such as Beyn's method, are generically \emph{not} eigenpairs of $T(\cdot)$.
In contrast, for NLFEAST, given an NEP satisfying \cref{asp:T}, there exists a $\gamma_{N}$ such that for any quadrature rule satisfying \cref{eq:quadrature}, the computed eigenpairs (linearly) converge to the exact eigenpairs of $T(\cdot)$ as the iterations proceed.
Moreover, for a prescribed $\gamma_{N}$, only a finite number of quadrature nodes is required to meet \cref{eq:quadrature}.

\section{Numerical experiments}
\label{sec:numexp}

In this section, we provide some numerical results to illustrate the convergence behavior and
efficiency of the framework described in~\cref{algo:iree}.
Throughout this section, we choose NLFEAST (\cref{algo:NLFEAST}) as a  concrete instance of~\cref{algo:iree}.
A discussion of other instances within the framework of~\cref{algo:iree}
is provided in~\cref{sec:nlfeast_mod}.
All numerical experiments were performed using MATLAB R2023b on a Linux server
with two 16-core Intel Xeon Gold 6226R 2.90 GHz CPUs and 1024~GB main memory.

\subsection{Test examples and experiment settings}

We consider the nine NEPs listed in~\cref{tab:expsinfo}, with the corresponding contours and eigenvalues of interest illustrated in~\cref{fig:9exp_region}.
Almost all are from the NLEVP collection~\cite{BHM2013}, except for
\texttt{photonics}, which arises from the computation of resonant frequencies for dispersive photonic materials and is described in~\cite{DAG2020,GMG2017}.

\begin{table}[htbp]
\centering
\caption{List of test problems.
QEP, PEP, REP, and NEP stand for quadratic, polynomial, rational and (general) nonlinear eigenvalue problems, respectively.
The integers $n$ and $m$ are the size of the problem and the number of eigenvalues of interest, respectively.
}
\begin{tabular}{c@{\hspace{0.8em}}c@{\hspace{0.8em}}c@{\hspace{0.8em}}c}
\toprule
Problem & Type & \(n\) & \(m\) \\
\midrule
{\tt{spring}}             & QEP  & \(3000 \) & \(32\)   \\
{\tt{acoustic\_wave\_2d}} & QEP  & \(9900 \) & \(10\) \\ 
{\tt{butterfly}}          & PEP  & \(5000 \) & \(9 \)  \\ 
{\tt{loaded\_string}}     & REP  & \(20000\) & \(10\)   \\ 
{\tt{photonics}}          & REP  & \(20363\) & \(16\)  \\ 
{\tt{railtrack2\_rep}}    & REP  & \(35955\) & \(2 \)\\ 
{\tt{hadeler}}            & NEP  & \(5000 \) & \(13\) \\ 
{\tt{gun}}                & NEP  & \(9956 \) & \(21\)  \\ 
{\tt{pdde\_symmetric}}    & NEP  & \(16281\) & \(5 \)  \\ 
\bottomrule
\end{tabular}
\label{tab:expsinfo}
\end{table}

\begin{figure}[tb!]
\centering
\includegraphics[width=0.9\linewidth]{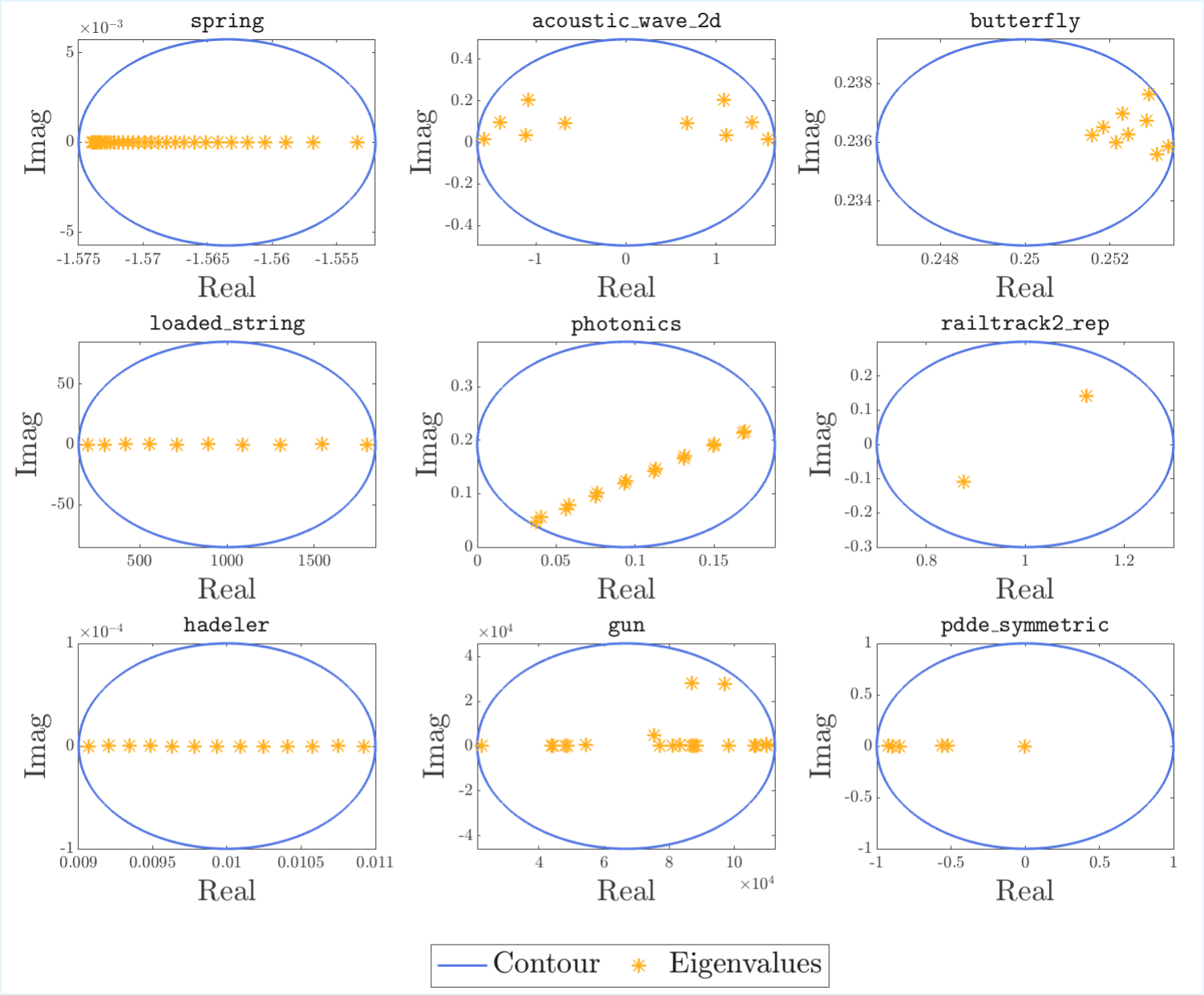}
\caption{The contour and eigenvalues of the test problems in~\cref{tab:expsinfo}.}
\label{fig:9exp_region}
\end{figure}

Following the representation adopted in the NLEVP collection, the test NEPs are given in the form
\[
T(\xi) = g_0(\xi)T_{0} + \cdots + g_p(\xi)T_p,
\]
where \(g_j(\xi)\colon\region\mapsto\C\) is a scalar-valued function and \(T_j\in\mathbb{C}^{n\times n}\) is a constant matrix for \(j=0\), \(\dotsc\), \(p\).
The (relative) residual of an approximate eigenpair \((\rho,x)\) is defined as
\[
r=\frac{\norm{T(\rho)x}}{\bigl(\sum_{j=0}^p\abs{g_{j}(\rho)}\norm{T_{j}}\bigr)\norm{x}}.
\]
At the \(k\)th iteration, suppose that all the computed approximate eigenpairs are
\((\rho_i^{(k)},x_i^{(k)})\) for \(1\le i\le m\),
then, the residual of this iteration is defined as
\[
r^{(k)}=\max_{1\leq i\leq m} r_i^{(k)}=
\max_{1\leq i\leq m} \frac{\norm{T(\rho_i^{(k)})x_i^{(k)}}}{\bigl(\sum_{j=0}^p\abs{g_{j}(\rho_i^{(k)})}\norm{T_{j}}\bigr)\norm{x_i^{(k)}}}.
\]
In this section, we use \(r^{(k)}\le10^{-12}\) as the stopping criterion for all numerical results reported.

In line~\ref{alg-step:proj} of \cref{algo:iree}, we need to ensure that the solution of the 
compressed NEP satisfies~\cref{asp:RR}.
Unless stated otherwise, we apply Beyn's method to solve the compressed NEP \(Q^{\Htran} T(\cdot)Q\), where $Q$ is an orthonormal basis of $\mathcal{X}^{(k)}$.
The number of quadrature nodes in Beyn's method is chosen sufficiently large to match the requirements of~\cref{asp:RR}.
Since the size of the compressed problem is much smaller than that of the original problem, such a choice does not incur significant additional computational cost.

\subsection{Linear convergence of NLFEAST}

We first illustrate the linear convergence behavior of NLFEAST for the NEPs from~\cref{tab:expsinfo}.
Using the trapezoidal rule with $N$ (as specified in \cref{tab:time}) quadrature nodes, the observed runtimes and convergence histories are reported in \cref{tab:time} and \cref{fig:9exp_conv}, respectively.
The linear convergence of NLFEAST is clearly visible for all nine test problems, validating our theoretical results from \cref{sec:conv}.

\begin{table}[htbp]
\centering

\caption{Computation time of NLFEAST (with  $N$ quadrature nodes) for solving the NEPs from~\cref{tab:expsinfo}.
The total runtime as well as the individual times consumed by
LU decompositions, filtering, and extraction
are listed in seconds (s).
}
\label{tab:time}
\begin{tabular}{c@{\hspace{0.8em}}c@{\hspace{0.8em}}c@{\hspace{0.8em}}c@{\hspace{0.8em}}c@{\hspace{0.8em}}c}
\toprule
Problem & $N$ & LU (s) & Filtering (s) & Extraction (s) & Total (s)\\
\midrule
{\tt{spring}}     &8       & $0.1526$ & $0.1108$ & $0.08253$ & $0.4191$\\
{\tt{acoustic\_wave\_2d}}&8 & $0.3831$ & $8.074$  & $0.9639$  & $10.13$ \\ 
{\tt{butterfly}}&16          & $0.3651$ & $3.763$  & $0.4453$  & $4.848$ \\ 
{\tt{loaded\_string}}&8     & $3.205$  & $1.406$  & $0.2706$  & $5.387$ \\ 
{\tt{photonics}}       &32   & $10.06$  & $142.1$  & $3.464$   & $157.8$ \\ 
{\tt{railtrack2\_rep}}    &16& $68.39$  & $50.94$  & $0.1413$  & $120.3$ \\ 
{\tt{hadeler}}            &8& $9.312$  & $3.119$  & $2.425$   & $19.24$ \\ 
{\tt{gun}}                &32& $11.32$  & $90.65$  & $2.37$    & $105.7$ \\ 
{\tt{pdde\_symmetric}}    &8& $0.5877$ & $4.536$  & $0.2085$  & $5.584$ \\ 
\bottomrule
\end{tabular}
\end{table}

\begin{figure}[tb!]
\centering
\includegraphics[width=0.9\linewidth]{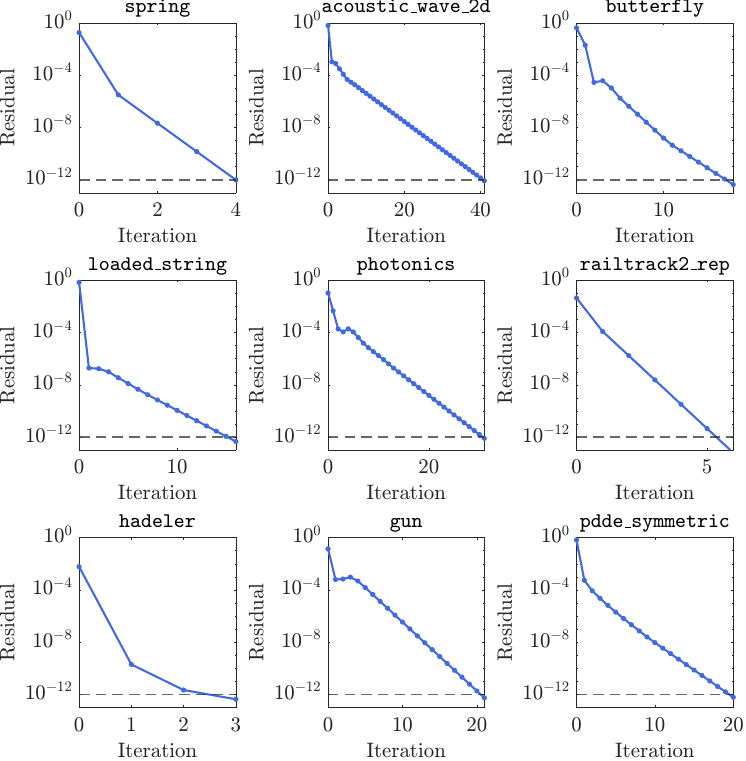}
\caption{Convergence history of NLFEAST for solving the NEPs from~\cref{tab:expsinfo}.}
\label{fig:9exp_conv}
\end{figure}

\subsection{Comparison between NLFEAST and Beyn's method}

The purpose of this section is to highlight the benefits of using a contour integral-based nonlinear eigensolver that allows for iterative refinement, such as NLFEAST. 

It is important to recall that the linear systems arising in NLFEAST involve matrices $T(\xi_{j})$, $j=1,\dotsc,N$, that remain the same throughout the iterations.
Thus, we just need to compute (sparse) LU factorizations of these $N$ matrices only once, during the first iteration, and can reuse these factorizations in later iterations. 

In contrast, the only way to improve the accuracy of Beyn's method is to increase the number of quadrature nodes.
A practical inconvenience is that one usually does not know in advance how many quadrature nodes are needed to reach a specified accuracy.
Even if \(N\) can be (tightly) estimated, achieving higher accuracy inevitably needs more quadrature nodes.
This leads to linear systems involving a larger number of distinct matrices and, in turn, requires computing a larger number of matrix factorizations.
Since matrix factorizations are usually more expensive than subsequent (triangular) linear solves, one expects NLFEAST to be faster than Beyn's method when desiring highly accurate eigenpair approximations.

In this numerical experiment, we select the three test problems \texttt{acoustic\_wave\_2d} (QEP), \texttt{photonics} (REP), and \texttt{gun} (NEP) from \cref{tab:expsinfo},
and compare the performance of NLFEAST with Beyn's method in \cref{fig:t2a}.
For NLFEAST, we fix the number of quadrature nodes to \(N=16\), \(32\), or \(48\), and plot the residual vs. the accumulated computational time after every iteration. For Beyn's method, the residual vs. the computational time is plotted when choosing $N=16+kd$ quadrature nodes for $k=1,2,\dotsc$, where the step sizes $d$ are set to $8$, $64$, and $16$ for the three problems, respectively. Note that NLFEAST with $N=16$ fails to converge on \texttt{photonics}.

As shown in~\cref{fig:t2a}, except when very few quadrature nodes are used (such as $N = 16$ for \texttt{photonics}), NLFEAST consistently outperforms Beyn's method, usually achieving the small residuals within less time. Furthermore, the convergence rate of NLFEAST is much more favorable, because its cost for sufficiently many iterations is governed primarily by the cost of triangular solves rather than matrix factorizations.

\begin{figure}[tb!]
\centering
\includegraphics[width=0.95\linewidth]{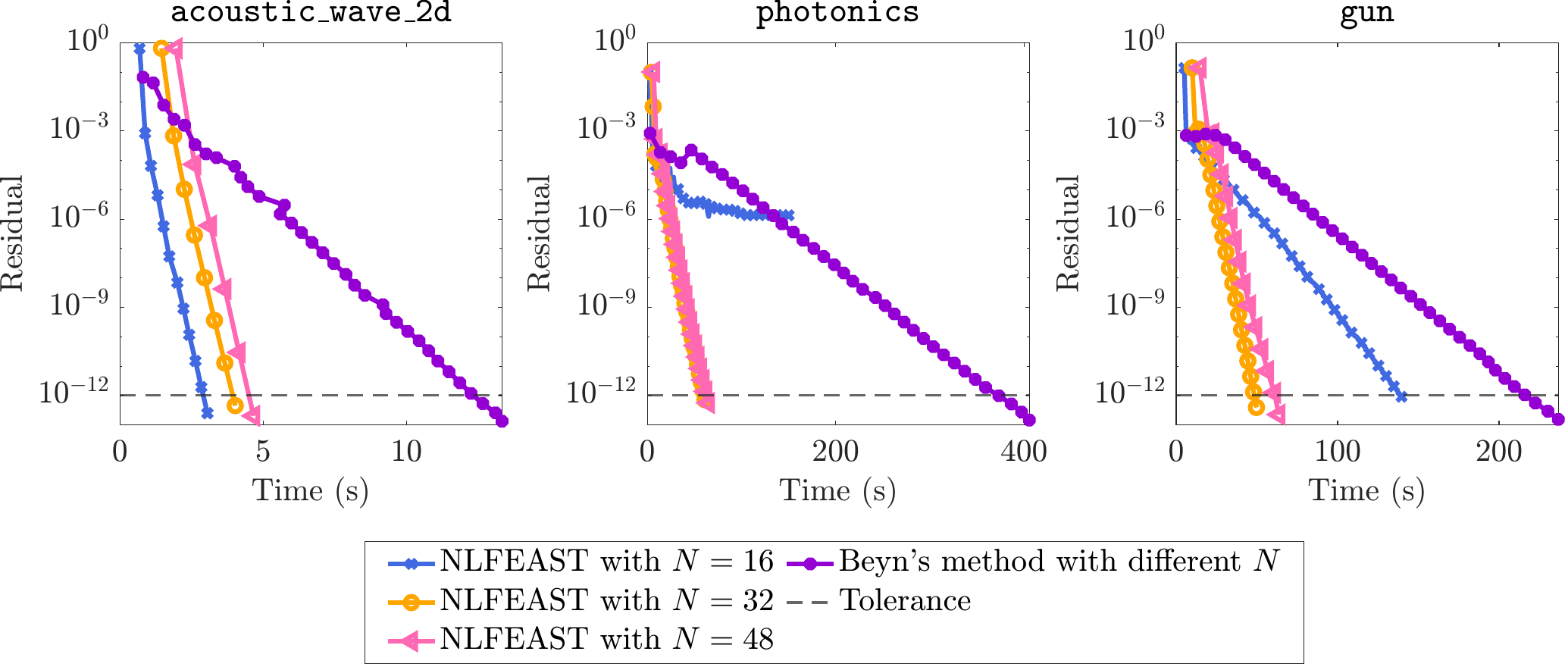}
\caption{Residual vs. computational time  for NLFEAST and Beyn's method for three test problems.}
\label{fig:t2a}
\end{figure}

\subsection{Special cases where
 different eigenvalues share the same eigenvector}
 \label{subsec:numexpSameEvec}

One important feature that distinguishes NEPs from linear eigenvalue problems is that eigenvectors belonging to different eigenvalues are not necessarily linearly independent.
This can be a hindrance when solving NEPs. For example, if one wants to avoid ghost eigenvalues by performing (re)orthogonalization as in Krylov subspace methods for linear eigenvalue problems, some eigenvalues could be missed due to linear dependencies among eigenvectors \cite{Kressner2009}.
For NLFEAST, once we solve the compressed NEP in a reliable way, such as companion linearization for polynomial/rational eigenvalue problems, the following numerical experiment shows that NLFEAST is not affected by this issue.

Consider the following $3$-dimensional quadratic eigenvalue problem
\begin{equation}
\label{eq:qep_daniel}
T(\xi)=
\bmat{0 & 12 & 0\\-2 & 14 & 0\\0 & 0 & 0}
+\xi\bmat{-1 & -6 & 0\\2 & -9 & 0\\0 & 0 & 0}
+\xi^2\bmat{1 & 0 & 0\\0 & 1 & 0\\0 & 0 & 1}
\end{equation}
inspired by~\cite[Eq.~(2)]{Kressner2009}, which has \(5\) different eigenvalues \(\lambda_0=0\) (with algebraic multiplicity \(2\)), \(\lambda_1=1\), \(\lambda_2=2\),
\(\lambda_3=3\), \(\lambda_4=4\) with corresponding left eigenvectors
\[
w_0=\bmat{0\\0\\1},\qquad
w_1=\bmat{1\\-1\\0},\qquad
w_2=\bmat{1\\-1\\0},\qquad
w_3=\bmat{2\\-3\\0},\qquad
w_4=\bmat{1\\-2\\0},
\]
and right eigenvectors
\[
v_0=\bmat{0\\0\\1},\qquad
v_1=\bmat{1\\0\\0},\qquad
v_2=\bmat{0\\1\\0},\qquad
v_3=\bmat{1\\1\\0},\qquad
v_4=\bmat{1\\1\\0},
\]
respectively.
We consider three different test contours $\mathcal{C}_{1}=\{\abs{z-1.5}=1\}$, $\mathcal{C}_{2}=\{\abs{z-2.5}=1\}$ and $\mathcal{C}_{3}=\{\abs{z-3.5}=1\}$ as in~\cref{fig:same_eigv}.
We note that $\mathcal{C}_{1}$ encloses eigenvalues with linearly independent right eigenvectors but same left eigenvectors,
$\mathcal{C}_{2}$ encloses eigenvalues with linearly independent left eigenvectors and linearly independent right eigenvectors,
while the eigenvalues enclosed by $\mathcal{C}_{3}$ share the same right eigenvectors but linearly independent left eigenvectors.

\begin{figure}[tb!]
\centering
\includegraphics[width=0.6\linewidth]{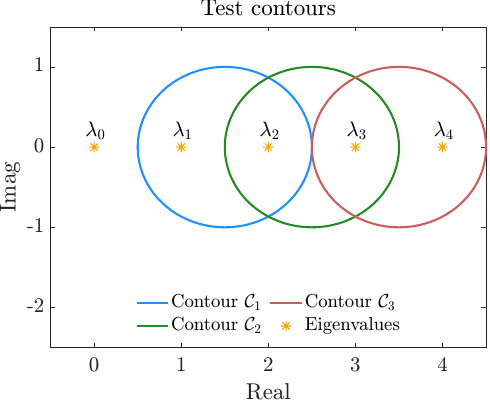}
\caption{The eigenvalues of the test example~\eqref{eq:qep_daniel}
and three test contours.
The eigenvalues $\lambda_{1}$ and $\lambda_{2}$ share the same left eigenvectors; and $\lambda_{3}$ and $\lambda_{4}$ share the same right eigenvectors.}
\label{fig:same_eigv}
\end{figure}

The eigenvalues computed by NLFEAST and Beyn's method on the three contours are reported in~\cref{tab:same_eigv}. The results indicate that Beyn’s method can handle only cases in which both the left and right eigenvectors associated with the eigenvalues inside the contour are linearly independent. In contrast, NLFEAST remains robust even when the left and right eigenvectors corresponding to different eigenvalues are identical; provided a reliable solver is used for the compressed NEP, NLFEAST is still able to produce accurate results.

\begin{table}[h]
\caption{Approximate eigenvalues and corresponding residual computed by NLFEAST and Beyn's method on the three test contours in~\cref{fig:same_eigv}. In NLFEAST, we solve the compressed NEP by companion linearization.}
\label{tab:same_eigv}
\centering
    \begin{tabular}{cccccc}
\toprule
\multicolumn{2}{c}{} & \multicolumn{2}{c}{NLFEAST} & \multicolumn{2}{c}{Beyn's method} \\
& & \(\check\lambda_1\) & \(\check\lambda_2\) & \(\check\lambda_1\) & \(\check\lambda_2\)\\ 
\midrule
\multirow{2}*{$\mathcal{C}_{1}$}
& Eigenvalue & \(1.000\) & \(2.000\) & \(0.457-0.174\cdot\mi\) & \(1.653+0.104\cdot\mi\)\\
& Residual & \(2.912\times10^{-16}\) & \(8.889\times10^{-17}\) & \(9.658\times10^{-1}\) & \(5.361\times10^{-1}\)\\
\midrule
\multirow{2}*{$\mathcal{C}_{2}$}
& Eigenvalue & \(2.000\) & \(3.000\) & \(2.000\) & \(3.000\)\\
& Residual & \(1.376\times10^{-16}\) & \(3.717\times10^{-17}\) & \(1.577\times10^{-14}\) & \(2.151\times10^{-15}\) \\
\midrule
\multirow{2}*{$\mathcal{C}_{3}$}
& Eigenvalue & \(3.000\) & \(4.000\) & \(2.015-0.073\cdot\mi\) & \(3.812+0.151\cdot\mi\)\\
& Residual & \(1.041\times10^{-16}\) & \(5.038\times10^{-17}\) & \(1.368\times10^{-2}\) & \(4.959\times10^{-1}\)\\ 
\bottomrule
    \end{tabular}
\end{table}

\section{Concluding remarks}

In this paper, we proposed a general framework for iterative contour integral-based methods for nonlinear eigenvalue problems. Under arguably mild assumptions, we establish linear convergence for methods from this framework, including NLFEAST. 
One promising direction for future research is to explore this framework further and design novel iterative nonlinear eigensolvers with improved properties.

\appendix

\section{Some variants of NLFEAST}
\label{sec:nlfeast_mod}

A natural question is whether there exist algorithms other than NLFEAST (\cref{algo:NLFEAST}) that fit
the framework of~\cref{algo:iree}.
For linear eigenvalue problems, a family of such algorithms can be constructed by choosing a different
filter \(\itf\), including, for example, subspace iteration, FEAST, and Riemannian gradient descent.
In contrast, for nonlinear eigenvalue problems, constructing a filter \(\itf\) that satisfies the
convergence conditions established in~\cref{sub-sec:suffcond} is rather challenging,
as can already be seen from the analysis presented in~\cref{sub-sec:NLFEAST}.
To the best of our knowledge, apart from NLFEAST, no other iterative refinement algorithms are currently known to conform to the framework of~\cref{algo:iree}.

For exploratory purposes, we include \cref{algo:NLFEAST_relaxed}, which is in the spirit of~\cref{algo:iree} but chooses the filter adaptively (and differently) in every iteration.
{The basic idea of~\cref{algo:NLFEAST_relaxed} is motivated by the observation that,
in each iteration of NLFEAST, it obtains the updated vector \(\widehat{x}_{i}^{(k)}\) from a fixed linear combination of $x_{i}^{(k)}$ and 
\[
P_{i}^{(k)}x_{i}^{(k)}\defi \sum_{j=1}^{N} \frac{\omega_{j}}{\xi_{j}-\rho_{i}^{(k)}}T(\xi_{j})^{-1}T(\rho_{i}^{(k)})x_{i}^{(k)}
\]
as \(x_{i}^{(k)} - Px_{i}^{(k)}\).
This suggests that seeking a better update within the subspace \(\Span\{x_{i}^{(k)}, P_{i}^{(k)}x_{i}^{(k)}\}\)
may accelerate convergence.
To this end, we attempt to approximate a Newton-like descent direction \((T(\rho_{i}^{(k)}))^{-1}T^{\prime}(\rho_{i}^{(k)})x_{i}^{(k)}\) within this subspace, resulting in the (simple) optimization problem of line~\ref{line:acc} for choosing the vector from the subspace.}
\begin{algorithm}[htbp]
\caption{Relaxed NLFEAST}
\label{algo:NLFEAST_relaxed}
\begin{algorithmic}[1]
\REQUIRE Matrix-valued function $T$, and an $m$-dimensional initial subspace $\mathcal{X}^{(0)}\subset \C^{n}$.
\ENSURE Ritz pairs $(\rho_{i},x_{i})$ for $1\leq i\leq m$.
\STATE Compute quadrature nodes and weights $(\xi_{j},\omega_{j})$ for $1\leq j\leq N$.
\FOR{\(k=0\), \(1\), \(\dotsc\)}
\STATE Extract Ritz pairs $(\rho_{i}^{(k)},x_{i}^{(k)})$ for $1\leq i\leq m$ in $\mathcal{X}^{(k)}$ by Rayleigh--Ritz procedure.
\FOR{\(i=1\), \(2\), \(\dotsc\), \(m\)}
\STATE Compute
\[
y_{i}^{(k)}=
\sum_{j=1}^{N} \frac{\omega_{j}}{\xi_{j}-\rho_{i}^{(k)}}\Bigl(T(\xi_{j})^{-1}T(\rho_{i}^{(k)})\Bigr)x_{i}^{(k)}.
\]
\STATE \label{line:acc} Compute \(\widehat{x}_{i}^{(k)}=\alpha x_{i}^{(k)}+\beta y_{i}^{(k)}\) where \(\alpha\) and \(\beta\) are the solution of
\[
\argmin_{\alpha, \beta\in\mathbb{C}}\bigl\lVert T(\rho_{i}^{(k)})(\alpha x_{i}^{(k)}+\beta y_{i}^{(k)})
-T'(\rho_{i}^{(k)})x_{i}^{(k)}\bigr\rVert.
\]
\ENDFOR
\STATE Form a new subspace $\mathcal{X}^{(k+1)}=\Span\{\widehat{x}_{1}^{(k)},\dotsc,\widehat{x}_{m}^{(k)}\}$.
\ENDFOR
\end{algorithmic}
\end{algorithm}

The convergence history of~\cref{algo:NLFEAST_relaxed} on the $9$ test examples are shown in~\cref{fig:nlfeast_relaxed}.
In line with being inspired by Newton's method, it only exhibits local convergence,
and for some test problems several iterations are required before the local convergence region is entered; see \texttt{butterfly} and \texttt{hadeler} in~\cref{fig:nlfeast_relaxed}.
Moreover, it should be emphasized that although~\cref{algo:NLFEAST_relaxed} can be viewed as a relaxation of NLFEAST,
its performance is not necessarily superior to NLFEAST, and it strongly depends on the problems.
In particular, \cref{algo:NLFEAST_relaxed} fails to converge for the \texttt{railtrack2\_rep} problem.
A possible modification is to enlarge the search space \(\Span\{x_{i}^{(k)}, P_{i}^{(k)}x_{i}^{(k)}\}\)
with \((P_{i}^{(k)})^jx_{i}^{(k)}\) for \(j=2\), \(3\), \(\dotsc\),
which will lead to an Arnoldi style method.
However, this modification is not always observed to perform better.
Further investigation related to~\cref{algo:NLFEAST_relaxed} is left for future work.
\begin{figure}[tb!]
\centering
\includegraphics[width=0.9\linewidth]{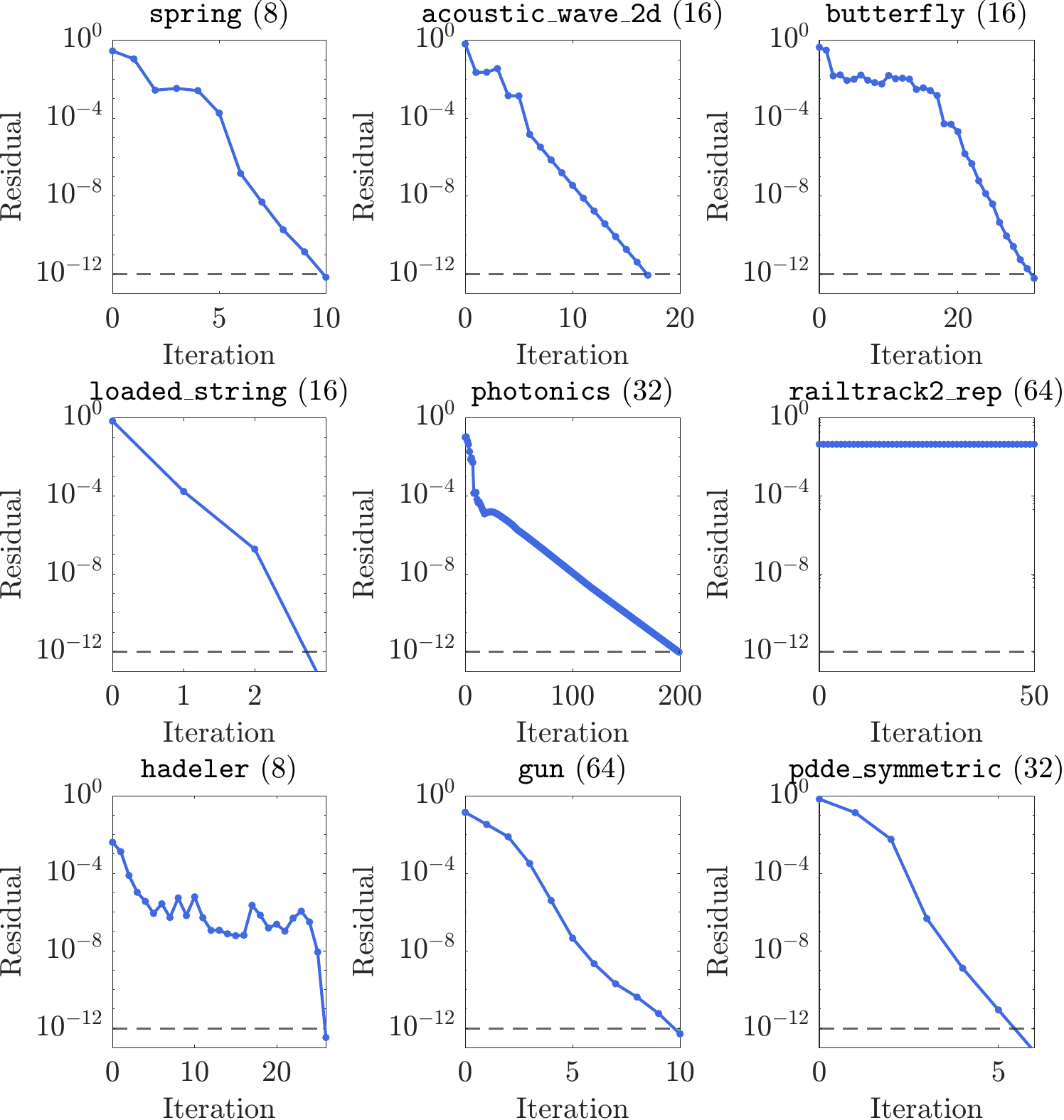}
\caption{Convergence history of \cref{algo:NLFEAST_relaxed} applied to the nine NEPs from~\cref{tab:expsinfo}.
The numbers in parenthesis denote the number of quadrature nodes (\(N\)) used in each example.}
\label{fig:nlfeast_relaxed}
\end{figure}

\section*{Statements and Declarations}
\begin{itemize}
	\item Funding: Jose E.~Roman was supported by grant PID2022-139568NB-I00 funded by MCIN/AEI/10.13039/501100011033 and by ERDF/EU.
	               Meiyue Shao was partly supported by the National Natural Science Foundation of China grant No.~92370105.
	\item Code availability: Scripts to reproduce numerical results are publicly available at \href{https://github.com/IzayoiYuuki/Linear_convergence_of_iterative_contour_integral-based_eigensolvers_for_nonlinear_eigenvalue_problem}{https://github.com/IzayoiYuuki/Linear\_convergence\_of\_iterative\_contour\_integral-based\_eigensolvers\_for\_nonlinear\_eigenvalue\_problem}
	\item Conflict of interests: Not applicable.
\end{itemize}



\begin{thebibliography}{34}
\ifx \bisbn   \undefined \def \bisbn  #1{ISBN #1}\fi
\ifx \binits  \undefined \def \binits#1{#1}\fi
\ifx \bauthor  \undefined \def \bauthor#1{#1}\fi
\ifx \batitle  \undefined \def \batitle#1{#1}\fi
\ifx \bjtitle  \undefined \def \bjtitle#1{#1}\fi
\ifx \bvolume  \undefined \def \bvolume#1{\textbf{#1}}\fi
\ifx \byear  \undefined \def \byear#1{#1}\fi
\ifx \bissue  \undefined \def \bissue#1{#1}\fi
\ifx \bfpage  \undefined \def \bfpage#1{#1}\fi
\ifx \blpage  \undefined \def \blpage #1{#1}\fi
\ifx \burl  \undefined \def \burl#1{\textsf{#1}}\fi
\ifx \doiurl  \undefined \def \doiurl#1{\url{https://doi.org/#1}}\fi
\ifx \betal  \undefined \def \betal{\textit{et al.}}\fi
\ifx \binstitute  \undefined \def \binstitute#1{#1}\fi
\ifx \binstitutionaled  \undefined \def \binstitutionaled#1{#1}\fi
\ifx \bctitle  \undefined \def \bctitle#1{#1}\fi
\ifx \beditor  \undefined \def \beditor#1{#1}\fi
\ifx \bpublisher  \undefined \def \bpublisher#1{#1}\fi
\ifx \bbtitle  \undefined \def \bbtitle#1{#1}\fi
\ifx \bedition  \undefined \def \bedition#1{#1}\fi
\ifx \bseriesno  \undefined \def \bseriesno#1{#1}\fi
\ifx \blocation  \undefined \def \blocation#1{#1}\fi
\ifx \bsertitle  \undefined \def \bsertitle#1{#1}\fi
\ifx \bsnm \undefined \def \bsnm#1{#1}\fi
\ifx \bsuffix \undefined \def \bsuffix#1{#1}\fi
\ifx \bparticle \undefined \def \bparticle#1{#1}\fi
\ifx \barticle \undefined \def \barticle#1{#1}\fi
\bibcommenthead
\ifx \bconfdate \undefined \def \bconfdate #1{#1}\fi
\ifx \botherref \undefined \def \botherref #1{#1}\fi
\ifx \url \undefined \def \url#1{\textsf{#1}}\fi
\ifx \bchapter \undefined \def \bchapter#1{#1}\fi
\ifx \bbook \undefined \def \bbook#1{#1}\fi
\ifx \bcomment \undefined \def \bcomment#1{#1}\fi
\ifx \oauthor \undefined \def \oauthor#1{#1}\fi
\ifx \citeauthoryear \undefined \def \citeauthoryear#1{#1}\fi
\ifx \endbibitem  \undefined \def \endbibitem {}\fi
\ifx \bconflocation  \undefined \def \bconflocation#1{#1}\fi
\ifx \arxivurl  \undefined \def \arxivurl#1{\textsf{#1}}\fi
\csname PreBibitemsHook\endcsname

\bibitem[\protect\citeauthoryear{Araujo~C. et~al.}{2020}]{CER2020}
\begin{barticle}
\bauthor{\bsnm{Araujo~C.}, \binits{J.C.}},
\bauthor{\bsnm{Campos}, \binits{C.}},
\bauthor{\bsnm{Engstr{\"o}m}, \binits{C.}},
\bauthor{\bsnm{Roman}, \binits{J.E.}}:
\batitle{Computation of scattering resonances in absorptive and dispersive media with applications to metal-dielectric nano-structures}.
\bjtitle{J. Comput.\ Phys.}
\bvolume{407},
\bfpage{109220}
(\byear{2020})
\doiurl{10.1016/j.jcp.2019.109220}
\end{barticle}
\endbibitem

\bibitem[\protect\citeauthoryear{Asakura et~al.}{2009}]{AST2009}
\begin{barticle}
\bauthor{\bsnm{Asakura}, \binits{J.}},
\bauthor{\bsnm{Sakurai}, \binits{T.}},
\bauthor{\bsnm{Tadano}, \binits{H.}},
\bauthor{\bsnm{Ikegami}, \binits{T.}},
\bauthor{\bsnm{Kimura}, \binits{K.}}:
\batitle{A numerical method for nonlinear eigenvalue problems using contour integrals}.
\bjtitle{JSIAM Lett.}
\bvolume{1},
\bfpage{52}--\blpage{55}
(\byear{2009})
\doiurl{10.14495/jsiaml.1.52}
\end{barticle}
\endbibitem

\bibitem[\protect\citeauthoryear{Baydoun et~al.}{2021}]{SMB2021}
\begin{barticle}
\bauthor{\bsnm{Baydoun}, \binits{S.K.}},
\bauthor{\bsnm{Voigt}, \binits{M.}},
\bauthor{\bsnm{Goderbauer}, \binits{B.}},
\bauthor{\bsnm{Jelich}, \binits{C.}},
\bauthor{\bsnm{Marburg}, \binits{S.}}:
\batitle{A subspace iteration eigensolver based on {C}auchy integrals for vibroacoustic problems in unbounded domains}.
\bjtitle{Int.\ J.\ Numer.\ Methods Eng.}
\bvolume{122}(\bissue{16}),
\bfpage{4250}--\blpage{4269}
(\byear{2021})
\doiurl{10.1002/nme.6701}
\end{barticle}
\endbibitem

\bibitem[\protect\citeauthoryear{Betcke and Voss}{2007}]{BV2007}
\begin{barticle}
\bauthor{\bsnm{Betcke}, \binits{M.M.}},
\bauthor{\bsnm{Voss}, \binits{H.}}:
\batitle{Stationary {S}chr{\"o}dinger equations governing electronic states of quantum dots in the presence of spin-orbit splitting}.
\bjtitle{Appl.\ Math.}
\bvolume{52},
\bfpage{267}--\blpage{284}
(\byear{2007})
\doiurl{10.1007/s10492-007-0014-5}
\end{barticle}
\endbibitem

\bibitem[\protect\citeauthoryear{Betcke et~al.}{2013}]{BHM2013}
\begin{barticle}
\bauthor{\bsnm{Betcke}, \binits{T.}},
\bauthor{\bsnm{Higham}, \binits{N.J.}},
\bauthor{\bsnm{Mehrmann}, \binits{V.}},
\bauthor{\bsnm{Schr\"{o}der}, \binits{C.}},
\bauthor{\bsnm{Tisseur}, \binits{F.}}:
\batitle{{NLEVP}: A collection of nonlinear eigenvalue problems}.
\bjtitle{ACM Trans.\ Math.\ Software}
\bvolume{39}(\bissue{2}),
\bfpage{1}--\blpage{28}
(\byear{2013})
\doiurl{10.1145/2427023.2427024}
\end{barticle}
\endbibitem

\bibitem[\protect\citeauthoryear{Beyn}{2012}]{Beyn2012}
\begin{barticle}
\bauthor{\bsnm{Beyn}, \binits{W.-J.}}:
\batitle{An integral method for solving nonlinear eigenvalue problems}.
\bjtitle{Linear Algebra Appl.}
\bvolume{436}(\bissue{10}),
\bfpage{3839}--\blpage{3863}
(\byear{2012})
\doiurl{10.1016/j.laa.2011.03.030}
\end{barticle}
\endbibitem

\bibitem[\protect\citeauthoryear{Brennan et~al.}{2023}]{BEG2023}
\begin{barticle}
\bauthor{\bsnm{Brennan}, \binits{M.C.}},
\bauthor{\bsnm{Embree}, \binits{M.}},
\bauthor{\bsnm{Gugercin}, \binits{S.}}:
\batitle{Contour integral methods for nonlinear eigenvalue problems: a systems theoretic approach}.
\bjtitle{SIAM Rev.}
\bvolume{65}(\bissue{2}),
\bfpage{439}--\blpage{470}
(\byear{2023})
\doiurl{10.1137/20M1389303}
\end{barticle}
\endbibitem

\bibitem[\protect\citeauthoryear{Brenneck and Polizzi}{2020}]{BP2020}
\begin{barticle}
\bauthor{\bsnm{Brenneck}, \binits{J.}},
\bauthor{\bsnm{Polizzi}, \binits{E.}}:
\batitle{An iterative method for contour-based nonlinear eigensolvers}.
\bjtitle{arXiv preprint arXiv:2007.03000}
(\byear{2020})
\doiurl{10.48550/arXiv.2007.03000}
\end{barticle}
\endbibitem

\bibitem[\protect\citeauthoryear{Bruno et~al.}{2026}]{BST2024}
\begin{barticle}
\bauthor{\bsnm{Bruno}, \binits{O.P.}},
\bauthor{\bsnm{Santana}, \binits{M.}},
\bauthor{\bsnm{Trefethen}, \binits{L.N.}}:
\batitle{Evaluation of resonances: Adaptivity and {AAA} rational approximation of randomly scalarized boundary integral resolvents}.
\bjtitle{SIAM J. Sci. Comput.}
\bvolume{48}(\bissue{3}),
\bfpage{1260}--\blpage{1283}
(\byear{2026})
\doiurl{10.1137/24M1690680}
\end{barticle}
\endbibitem

\bibitem[\protect\citeauthoryear{Dem{\'e}sy et~al.}{2020}]{DAG2020}
\begin{barticle}
\bauthor{\bsnm{Dem{\'e}sy}, \binits{G.}},
\bauthor{\bsnm{Nicolet}, \binits{A.}},
\bauthor{\bsnm{Gralak}, \binits{B.}},
\bauthor{\bsnm{Geuzaine}, \binits{C.}},
\bauthor{\bsnm{Campos}, \binits{C.}},
\bauthor{\bsnm{Roman}, \binits{J.E.}}:
\batitle{Non-linear eigenvalue problems with {GetDP} and {SLEPc}: Eigenmode computations of frequency-dispersive photonic open structures}.
\bjtitle{Comput.\ Phys.\ Commun.}
\bvolume{257},
\bfpage{107509}
(\byear{2020})
\doiurl{10.1016/j.cpc.2020.107509}
\end{barticle}
\endbibitem

\bibitem[\protect\citeauthoryear{Garcia-Vergara et~al.}{2017}]{GMG2017}
\begin{barticle}
\bauthor{\bsnm{Garcia-Vergara}, \binits{M.}},
\bauthor{\bsnm{Dem{\'e}sy}, \binits{G.}},
\bauthor{\bsnm{Zolla}, \binits{F.}}:
\batitle{Extracting an accurate model for permittivity from experimental data: hunting complex poles from the real line}.
\bjtitle{Opt.\ Lett.}
\bvolume{42}(\bissue{6}),
\bfpage{1145}--\blpage{1148}
(\byear{2017})
\doiurl{10.1364/OL.42.001145}
\end{barticle}
\endbibitem

\bibitem[\protect\citeauthoryear{Gavin et~al.}{2018}]{GMP2018}
\begin{barticle}
\bauthor{\bsnm{Gavin}, \binits{B.}},
\bauthor{\bsnm{Mi{\k{e}}dlar}, \binits{A.}},
\bauthor{\bsnm{Polizzi}, \binits{E.}}:
\batitle{{FEAST} eigensolver for nonlinear eigenvalue problems}.
\bjtitle{J. Comput.\ Sci.}
\bvolume{27},
\bfpage{107}--\blpage{117}
(\byear{2018})
\doiurl{10.1016/j.jocs.2018.05.006}
\end{barticle}
\endbibitem

\bibitem[\protect\citeauthoryear{Ghani and Polifke}{2021}]{GP2021}
\begin{barticle}
\bauthor{\bsnm{Ghani}, \binits{A.}},
\bauthor{\bsnm{Polifke}, \binits{W.}}:
\batitle{An exceptional point switches stability of a thermoacoustic experiment}.
\bjtitle{J. Fluid Mech.}
\bvolume{920},
\bfpage{3}
(\byear{2021})
\doiurl{10.1017/jfm.2021.480}
\end{barticle}
\endbibitem

\bibitem[\protect\citeauthoryear{Golub and Van~Loan}{2013}]{GV2013}
\begin{bbook}
\bauthor{\bsnm{Golub}, \binits{G.H.}},
\bauthor{\bsnm{Van~Loan}, \binits{C.F.}}:
\bbtitle{Matrix Computations},
\bedition{4}th edn.
\bpublisher{Johns Hopkins University Press},
\blocation{Baltimore, MD, USA}
(\byear{2013})
\end{bbook}
\endbibitem

\bibitem[\protect\citeauthoryear{G{\"u}ttel and Tisseur}{2017}]{GT2017}
\begin{barticle}
\bauthor{\bsnm{G{\"u}ttel}, \binits{S.}},
\bauthor{\bsnm{Tisseur}, \binits{F.}}:
\batitle{The nonlinear eigenvalue problem}.
\bjtitle{Acta Numer.}
\bvolume{26},
\bfpage{1}--\blpage{94}
(\byear{2017})
\doiurl{10.1017/S0962492917000034}
\end{barticle}
\endbibitem

\bibitem[\protect\citeauthoryear{G\"uttel et~al.}{2024}]{GKV2024}
\begin{barticle}
\bauthor{\bsnm{G\"uttel}, \binits{S.}},
\bauthor{\bsnm{Kressner}, \binits{D.}},
\bauthor{\bsnm{Vandereycken}, \binits{B.}}:
\batitle{Randomized sketching of nonlinear eigenvalue problems}.
\bjtitle{SIAM J. Sci. Comput.}
\bvolume{46}(\bissue{5}),
\bfpage{3022}--\blpage{3043}
(\byear{2024})
\doiurl{10.1137/22M153656X}
\end{barticle}
\endbibitem

\bibitem[\protect\citeauthoryear{G\"uttel et~al.}{2022}]{GNT2022}
\begin{barticle}
\bauthor{\bsnm{G\"uttel}, \binits{S.}},
\bauthor{\bsnm{Negri~Porzio}, \binits{G.M.}},
\bauthor{\bsnm{Tisseur}, \binits{F.}}:
\batitle{Robust rational approximations of nonlinear eigenvalue problems}.
\bjtitle{SIAM J. Sci. Comput.}
\bvolume{44}(\bissue{4}),
\bfpage{2439}--\blpage{2463}
(\byear{2022})
\doiurl{10.1137/20M1380533}
\end{barticle}
\endbibitem

\bibitem[\protect\citeauthoryear{G\"uttel et~al.}{2014}]{SRK2014}
\begin{barticle}
\bauthor{\bsnm{G\"uttel}, \binits{S.}},
\bauthor{\bsnm{Van~Beeumen}, \binits{R.}},
\bauthor{\bsnm{Meerbergen}, \binits{K.}},
\bauthor{\bsnm{Michiels}, \binits{W.}}:
\batitle{N{LEIGS}: a class of fully rational {K}rylov methods for nonlinear eigenvalue problems}.
\bjtitle{SIAM J. Sci. Comput.}
\bvolume{36}(\bissue{6}),
\bfpage{2842}--\blpage{2864}
(\byear{2014})
\doiurl{10.1137/130935045}
\end{barticle}
\endbibitem

\bibitem[\protect\citeauthoryear{Huang et~al.}{2016}]{HSS2016}
\begin{barticle}
\bauthor{\bsnm{Huang}, \binits{R.}},
\bauthor{\bsnm{Struthers}, \binits{A.A.}},
\bauthor{\bsnm{Sun}, \binits{J.}},
\bauthor{\bsnm{Zhang}, \binits{R.}}:
\batitle{Recursive integral method for transmission eigenvalues}.
\bjtitle{J. Comput.\ Phys.}
\bvolume{327},
\bfpage{830}--\blpage{840}
(\byear{2016})
\doiurl{10.1016/j.jcp.2016.10.001}
\end{barticle}
\endbibitem

\bibitem[\protect\citeauthoryear{Keldysh}{1971}]{Keldysh1971}
\begin{barticle}
\bauthor{\bsnm{Keldysh}, \binits{M.V.}}:
\batitle{On the completeness of the eigenfunctions of some classes of non-selfadjoint linear operators}.
\bjtitle{Russian Mathematical Surveys}
\bvolume{26}(\bissue{4}),
\bfpage{15}--\blpage{44}
(\byear{1971})
\doiurl{10.1070/RM1971v026n04ABEH003985}
\end{barticle}
\endbibitem

\bibitem[\protect\citeauthoryear{Kressner}{2009}]{Kressner2009}
\begin{barticle}
\bauthor{\bsnm{Kressner}, \binits{D.}}:
\batitle{A block {N}ewton method for nonlinear eigenvalue problems}.
\bjtitle{Numer.\ Math.}
\bvolume{114}(\bissue{2}),
\bfpage{355}--\blpage{372}
(\byear{2009})
\doiurl{10.1007/s00211-009-0259-x}
\end{barticle}
\endbibitem

\bibitem[\protect\citeauthoryear{Lalanne et~al.}{2019}]{LYG2019}
\begin{barticle}
\bauthor{\bsnm{Lalanne}, \binits{P.}},
\bauthor{\bsnm{Yan}, \binits{W.}},
\bauthor{\bsnm{Gras}, \binits{A.}},
\bauthor{\bsnm{Sauvan}, \binits{C.}},
\bauthor{\bsnm{Hugonin}, \binits{J.-P.}},
\bauthor{\bsnm{Besbes}, \binits{M.}},
\bauthor{\bsnm{Dem{\'e}sy}, \binits{G.}},
\bauthor{\bsnm{Truong}, \binits{M.D.}},
\bauthor{\bsnm{Gralak}, \binits{B.}},
\bauthor{\bsnm{Zolla}, \binits{F.}},
\bauthor{\bsnm{Nicolet}, \binits{A.}},
\bauthor{\bsnm{Binkowski}, \binits{F.}},
\bauthor{\bsnm{Zschiedrich}, \binits{L.}},
\bauthor{\bsnm{Burger}, \binits{S.}},
\bauthor{\bsnm{Zimmerling}, \binits{J.}},
\bauthor{\bsnm{Remis}, \binits{R.}},
\bauthor{\bsnm{Urbach}, \binits{P.}},
\bauthor{\bsnm{Liu}, \binits{H.T.}},
\bauthor{\bsnm{Weiss}, \binits{T.}}:
\batitle{Quasinormal mode solvers for resonators with dispersive materials}.
\bjtitle{J. Opt.\ Soc.\ Am.\ A}
\bvolume{36}(\bissue{4}),
\bfpage{686}--\blpage{704}
(\byear{2019})
\doiurl{10.1364/JOSAA.36.000686}
\end{barticle}
\endbibitem

\bibitem[\protect\citeauthoryear{Li and Polizzi}{2025}]{LE2025}
\begin{barticle}
\bauthor{\bsnm{Li}, \binits{D.}},
\bauthor{\bsnm{Polizzi}, \binits{E.}}:
\batitle{Nonlinear eigenvalue algorithm for {$GW$} quasiparticle equations}.
\bjtitle{Phys. Rev. B}
\bvolume{111},
\bfpage{045137}
(\byear{2025})
\doiurl{10.1103/PhysRevB.111.045137}
\end{barticle}
\endbibitem

\bibitem[\protect\citeauthoryear{Mehrmann and Voss}{2004}]{MV2004}
\begin{barticle}
\bauthor{\bsnm{Mehrmann}, \binits{V.}},
\bauthor{\bsnm{Voss}, \binits{H.}}:
\batitle{Nonlinear eigenvalue problems: A challenge for modern eigenvalue methods}.
\bjtitle{GAMM-Mitteilungen}
\bvolume{27}(\bissue{2}),
\bfpage{121}--\blpage{152}
(\byear{2004})
\doiurl{10.1002/gamm.201490007}
\end{barticle}
\endbibitem

\bibitem[\protect\citeauthoryear{Neumaier}{1985}]{Neumaier1985}
\begin{barticle}
\bauthor{\bsnm{Neumaier}, \binits{A.}}:
\batitle{Residual inverse iteration for the nonlinear eigenvalue problem}.
\bjtitle{SIAM J. Numer. Anal.}
\bvolume{22}(\bissue{5}),
\bfpage{914}--\blpage{923}
(\byear{1985})
\doiurl{10.1137/0722055}
\end{barticle}
\endbibitem

\bibitem[\protect\citeauthoryear{Polizzi}{2009}]{Polizzi2009}
\begin{barticle}
\bauthor{\bsnm{Polizzi}, \binits{E.}}:
\batitle{Density-matrix-based algorithms for solving eigenvalue problems}.
\bjtitle{Phys.\ Rev.~B}
\bvolume{79},
\bfpage{115112}
(\byear{2009})
\doiurl{10.1103/physrevb.79.115112}
\end{barticle}
\endbibitem

\bibitem[\protect\citeauthoryear{Saad et~al.}{2019}]{SEM2019}
\begin{barticle}
\bauthor{\bsnm{Saad}, \binits{Y.}},
\bauthor{\bsnm{El-Guide}, \binits{M.}},
\bauthor{\bsnm{Mi{\k{e}}dlar}, \binits{A.}}:
\batitle{A rational approximation method for the nonlinear eigenvalue problem}.
\bjtitle{arXiv preprint arXiv:1901.01188}
(\byear{2019})
\doiurl{10.48550/arXiv.1901.01188}
\end{barticle}
\endbibitem

\bibitem[\protect\citeauthoryear{Sakurai and Sugiura}{2003}]{SS2003}
\begin{barticle}
\bauthor{\bsnm{Sakurai}, \binits{T.}},
\bauthor{\bsnm{Sugiura}, \binits{H.}}:
\batitle{A projection method for generalized eigenvalue problems using numerical integration}.
\bjtitle{J. Comput.\ Appl.\ Math.}
\bvolume{159}(\bissue{1}),
\bfpage{119}--\blpage{128}
(\byear{2003})
\doiurl{10.1016/S0377-0427(03)00565-X}
\end{barticle}
\endbibitem

\bibitem[\protect\citeauthoryear{Shao}{2026}]{shao2026stabilizing}
\begin{barticle}
\bauthor{\bsnm{Shao}, \binits{N.}}:
\batitle{Stabilizing the {R}ayleigh--{R}itz procedure by randomization}.
\bjtitle{arXiv preprint arXiv:2604.01037}
(\byear{2026})
\doiurl{10.48550/arXiv.2604.01037}
\end{barticle}
\endbibitem

\bibitem[\protect\citeauthoryear{Tang and Polizzi}{2014}]{TP2014}
\begin{barticle}
\bauthor{\bsnm{Tang}, \binits{P.T.P.}},
\bauthor{\bsnm{Polizzi}, \binits{E.}}:
\batitle{{FEAST} as a subspace iteration eigensolver accelerated by approximate spectral projection}.
\bjtitle{SIAM J.\ Matrix Anal.\ Appl.}
\bvolume{35}(\bissue{2}),
\bfpage{354}--\blpage{390}
(\byear{2014})
\doiurl{10.1137/13090866X}
\end{barticle}
\endbibitem

\bibitem[\protect\citeauthoryear{Tisseur and Higham}{2001}]{TH2001}
\begin{barticle}
\bauthor{\bsnm{Tisseur}, \binits{F.}},
\bauthor{\bsnm{Higham}, \binits{N.J.}}:
\batitle{Structured pseudospectra for polynomial eigenvalue problems, with applications}.
\bjtitle{SIAM J.\ Matrix Anal.\ Appl.}
\bvolume{23}(\bissue{1}),
\bfpage{187}--\blpage{208}
(\byear{2001})
\doiurl{10.1137/S0895479800371451}
\end{barticle}
\endbibitem

\bibitem[\protect\citeauthoryear{Trefethen}{2013}]{Trefethen2013}
\begin{bbook}
\bauthor{\bsnm{Trefethen}, \binits{L.N.}}:
\bbtitle{Approximation Theory and Approximation Practice}.
\bpublisher{Society for Industrial and Applied Mathematics (SIAM)},
\blocation{Philadelphia, PA, USA}
(\byear{2013}).
\doiurl{10.1137/1.9781611975949}
\end{bbook}
\endbibitem

\bibitem[\protect\citeauthoryear{Van~Beeumen et~al.}{2015}]{RKW2015}
\begin{barticle}
\bauthor{\bsnm{Van~Beeumen}, \binits{R.}},
\bauthor{\bsnm{Meerbergen}, \binits{K.}},
\bauthor{\bsnm{Michiels}, \binits{W.}}:
\batitle{Compact rational {K}rylov methods for nonlinear eigenvalue problems}.
\bjtitle{SIAM J. Matrix Anal. Appl.}
\bvolume{36}(\bissue{2}),
\bfpage{820}--\blpage{838}
(\byear{2015})
\doiurl{10.1137/140976698}
\end{barticle}
\endbibitem

\bibitem[\protect\citeauthoryear{Zhu and Knyazev}{2013}]{Zhu2013}
\begin{barticle}
\bauthor{\bsnm{Zhu}, \binits{P.}},
\bauthor{\bsnm{Knyazev}, \binits{A.V.}}:
\batitle{Angles between subspaces and their tangents}.
\bjtitle{J. Numer. Math.}
\bvolume{21}(\bissue{4}),
\bfpage{325}--\blpage{340}
(\byear{2013})
\doiurl{10.1515/jnum-2013-0013}
\end{barticle}
\endbibitem

\end{thebibliography}
\end{document}